\documentclass[12 pt]{elsarticle}
\usepackage{setspace}
\usepackage{geometry}
\usepackage{hyperref}
\usepackage{amsmath}
\usepackage{amssymb}
\usepackage{amsthm}
\usepackage{graphicx}
\usepackage{algorithmic}
\usepackage{algorithm}
\usepackage{caption}
\usepackage{subcaption}

\newtheorem{theorem}{Theorem}{}
\newtheorem{corollary}{Corollary}{}
\newtheorem{remark}{Remark}{}
\newtheorem{lemma}{Lemma}{}
\newtheorem{proposition}{Proposition}{}

\newcommand \E{\mathbb{E}}
\newcommand \Esp[1]{{\E\left[#1\right]}}
\newcommand{\Real}{{\Re}}
\newcommand{\Imag}{{\Im}}
\newcommand{\busSet}{\mathcal{B}}

\newcommand{\timeSet}{\mathcal{T}}

\newcommand{\vlin}{v^{\textrm{Lin}}}
\newcommand{\Slin}{S^{\textrm{Lin}}}
\newcommand{\vlinbar}{\bar{v}^{\textrm{Lin}}}
\newcommand{\Slinbar}{\bar{S}^{\textrm{Lin}}}
\newcommand{\slin}{s^{\textrm{Lin}}}
\newcommand{\Plin}{P^{\textrm{Lin}}}
\newcommand{\Qlin}{Q^{\textrm{Lin}}}
\newcommand{\dd}{{\textrm{d}}}

\journal{Electric Power Systems Research}

\begin{document}

\begin{frontmatter}

\title{Multi-stage Stochastic Alternating Current Optimal Power Flow with Storage: Bounding the Relaxation Gap}


\author[mymainaddress,mysecondaryaddress]{Maxime Grangereau\corref{mycorrespondingauthor}}
\ead{maxime.grangereau@gmail.com; maxime.grangereau@edf.fr; maxime.grangereau@polytechnique.edu}
\cortext[mycorrespondingauthor]{Corresponding author}

\author[mysecondaryaddress]{Wim van Ackooij}
\ead{wim.van-ackooij@edf.fr}

\author[mymainaddress,mythirdaddress]{Stéphane Gaubert}
\ead{stephane.gaubert@inria.fr}

\address[mymainaddress]{Ecole polytechnique, IP Paris, CMAP, CNRS, Route de Saclay, 91300 Palaiseau}
\address[mysecondaryaddress]{EDF Lab Saclay, Route de Saclay 91300 Palaiseau}
\address[mythirdaddress]{INRIA Saclay, Route de Saclay, 91300 Palaiseau}

\begin{abstract}
  We propose a generic multistage stochastic model for the Alternating Current Optimal Power Flow (AC OPF) problem {for radial distribution networks}, to account for the random electricity production of renewable energy sources and dynamic constraints of storage systems. We consider single-phase radial networks. Radial three-phase balanced networks (medium-voltage distribution networks typically have this structure) reduce to the former case. 
  This induces a large scale optimization problem, which, given the non-convex nature of the AC OPF, is generally challenging to solve to global optimality. We derive a priori conditions guaranteeing a vanishing relaxation gap for the multi-stage AC OPF problem, which can thus be solved using convex optimization algorithms. We also give an a posteriori upper bound on the relaxation gap. In particular, we show that a null or low relaxation gap may be expected for applications with light reverse power flows or if sufficient storage capacities with low cost are available. 
Then, we discuss the validity of our results when incorporating voltage regulation devices. Finally, we illustrate our results on problems of planning of a realistic distribution feeder with distributed solar production and storage systems. 
Scenario trees for solar production are constructed from a stochastic model, by a quantile-based algorithm.
\end{abstract}

\begin{keyword}
Optimal Power Flow, multistage stochastic optimization, convex relaxation, second-order cone programming, scenario trees
\end{keyword}

\end{frontmatter}


\section{Introduction}
\subsection{Motivation}
Distribution networks are currently facing a major change owing to the increasing share of decentralized Renewable Energy Sources (RES). They can create local physical violations and induce uncertainty on the network operating point: their production levels are random, as a result of the weather. To cope with these issues, many Distribution Network Operators (DNOs) are required to become able to respond locally to unforeseen events. This can be done by the means of energy flexibilities located at the nodes of the network, like storage systems. The dynamical and sometimes uncertain nature of these subsystems as well as the uncertainty brought by RES should be taken into account in the operation planning tools used by DNOs. One of these relevant tools is the so-called Optimal Power Flow (OPF) problem. It is a mathematical optimization problem which aims at finding an operating point of a power network that minimizes a given objective function, such as generation costs, active power losses, subject to constraints on power injections and losses, voltage magnitudes and intensities in the lines. 
We are interested in the Alternating Current Optimal Power Flow (AC OPF) problem which is an accurate physical model of the operating point of a network. It is non-convex and traditionally used in a static deterministic framework. However, as argued before, a {\em stochastic} and {\em dynamic} framework with storage systems and RES is of practical interest and this is the focus of the present paper.

\subsection{Optimal Power Flow Problem}
There exist two main declinations of the AC OPF problem: the Bus Injection Model (BIM), and the Branch Flow Model (BFM), both presented in \cite{low14_1}.
Both formulations are equivalent for radial connected networks \cite{ding:19}. In this paper, we choose the BFM formulation since it makes our arguments easier to present. The AC OPF problem is a non-convex optimization problem, recently shown to be strongly NP-hard \cite{bien:19}. 
However, a series of recent works initiated by the seminal paper \cite{lava:11:1} have shown that many real-world instances can be solved to global optimality using convex relaxations of the original problem.


A very extensive overview on relaxations and approximations of AC Power flow equations can be found in the book \cite{molz:19},
while a particular focus on conic relaxations is given in \cite{zohr:josz:jin:mada:lava:sojo:20}. We focus here on results on the conic relaxations of AC OPF problems for single phase radial networks. These models allow to treat the case of balanced three-phase radial networks, {like medium voltage distribution networks (20 kV) in France}, which reduces to the former case. The most famous conic relaxations of the AC OPF Problem are the Second-Order Cone (SOC) relaxation and the Semi-Definite (SD) relaxation. Both are presented in details in \cite{low14_1}. For radial networks, the SOC and SD relaxations are equivalent, but the SOC relaxation exhibits better numerical performance, and is preferable for such topologies. 
Conic relaxations are often used in the literature owing to their enhanced numerical performance and for the certificate of optimality they may provide, compared with the non-convex formulation. 
They are exact for many practical instances of OPF problem \cite{lava:11:1}. For this reason, many authors have studied exactness conditions for these convex relaxations. A posteriori conditions on the solution of the dual problem are given in \cite{lava:11:1}. 
The relaxation is exact for radial connected network under assumptions like over-generation \cite{lava:12}, load over-satisfaction \cite{fari:13}, or no upper bounds on voltage magnitude \cite{gan:12}. 
As these conditions are not verified by practical instances, other authors \cite{gan:14,huan:16} have obtained more realistic a priori exactness conditions.

\subsection{Contributions}

In this paper, we develop a generic model for the multi-stage stochastic version of the AC Optimal Power Flow problem with storage systems and intermittent RES. Multistage stochastic models allow to ensure the non-anticipativity of decision variables and are a pre-requisite in order to prevent decisions to depend on the realization of yet unknown data. These models require the formulation of a scenario tree, the size of which, has to grow exponentially with the number of time stages \cite{pflu:pich:14}. 
The multistage stochastic AC OPF problem is therefore a non-convex and large-scale problem. To alleviate the non-convexity issue, we consider the SOC relaxation of the problem. Our main contribution is to show that approaches guaranteeing the absence of relaxation gap for the AC OPF problem, originally developed in a static deterministic setting, can be extended to the multi-stage stochastic setting. Inspired by the approach of \cite{huan:16} in the deterministic case, we propose to restrict the feasible set of the problem by adding a finite number of linear constraints, which impose feasibility of a linearized power flow and compensations for active or reactive reverse power flows in the network. We show that the restricted problem has the same optimal value as the original one under realistic assumptions which can easily be checked a priori, see Proposition \ref{prop:suff:cond:equal:p:restricted}. The restricted problem has no relaxation gap, and hence its optimal value is easily computable, see Theorem \ref{theorem:opf:zero:duality:gap:restricted}. This allows us to easily compute a feasible solution of the original problem and an a posteriori upper bound on its relaxation gap, see Theorem \ref{theorem:a:posteriori:bound:relaxation:gap}. Besides, the result provides realistic and tractable a priori conditions guaranteeing that the relaxation gap of the original problem is zero, see Theorem \ref{theorem:opf:zero:duality:gap}. {Validity of our result when considering voltage regulation devices such as uncontrollable Voltage Regulation Transformers, SVCs and STATCOMs is also shown.} The interest of {our results} is numerically illustrated on a realistic distribution network with 56 buses found in \cite{fari:neal:clar:low:12}, equipped with distributed storage and solar panels. We generate the scenario trees for solar production by a quantile-based algorithm, based on a stochastic model of solar irradiance.

\subsection{Related work}
By comparison with other works guaranteeing zero relaxation gap for the AC OPF problem \cite{lava:11:1,lava:12,fari:13,gan:12,gan:14,huan:16}, we consider a multistage stochastic setting. The conditions given in our paper extend to the multistage stochastic case and to more general cost functionals the realistic zero relaxation gap conditions given in \cite{huan:16}. Moreover, we use this approach to provide a posteriori bounds on the relaxation gap.

Several works consider a deterministic dynamic AC OPF model, like \cite{grov:18}, which considers the SOC relaxation.

Some recent research proposes probabilistic (indifferently called chance-constrained) versions of the OPF problem. A probabilistic AC OPF model linearized around a reference scenario is studied in \cite{roal:17:2}. Other works consider the robust counterpart of the SD relaxation restricted to affine-linear decision-rules \cite{vrak:13}. A Semi-Definite convex relaxation of the chance-constrained AC OPF problem is proposed in \cite{venz:hali:mark:17} using a scenario based-approach and assuming piece-wise linear decision rules, or assuming Gaussian uncertainty. A Second-Order Cone approximation of the chance-constrained AC OPF is proposed in \cite{hali:chatz:pins:18} 
which allows good numerical performances, combined with a feasibility recovery method. None of these references account for storage systems, nor do they consider dynamical aspects of the problem.

Several other works deal with dynamic stochastic models. In particular, an important requirement for such problems is to ensure that decision variables remain non-anticipative, i.e., do not depend on yet unobserved random data. References \cite{swam:17} considers the simplified case where decision are non-anticipative for the initial time steps only. Non-anticipativity is guaranteed in \cite{jabr:14} using affine-linear policies but with a (linear) DC OPF model. Non-anticipativity of the decisions is also guaranteed in \cite{sun:16} which considers an iterative procedure to optimize a decision policy. By comparison, the approach by scenario trees developed here accounts both for non-anticipativity constraints and for the nonlinearity of the stochastic OPF problem, whereas it benefits from theoretical convergence guarantees. Indeed, scenario tree methods provide an approximation of the value of the original problem with continuous distribution of the random data, and this approximation converges to this value when the number of scenarios goes to infinity \cite{pflu:pich:14}.

\subsection{Outline of the paper}
This paper is organized as follows. Section \ref{sec:opf:model} presents the multi-stage stochastic AC OPF problem. Section \ref{sec:opf:restrict} introduces a restriction of this problem and gives conditions ensuring equality of the values of the original and restricted problems. Section \ref{sec:opf:np:relax:gap} presents the main result of this paper: the restricted problem has no relaxation gap. This provides a convenient way to establish a priori exactness of the SOC relaxation for particular instances of the original problem, or an easily computable bound on the relaxation gap. {Section \ref{sec:discussion:real:networks} discusses validity and possible extensions of our result when incorporating voltage regulation devices or when considering multi-phase unbalanced networks.} We illustrate numerically our results on a realistic distribution networks with 56 buses in Section \ref{sec:opf:numerics}.

\subsection{Notation}
The set of real numbers is denoted by $\mathbb{R}$ and the set of complex numbers by $\mathbb{C}$. The notation $\textbf{i}$ stands for the purely imaginary number such that $\textbf{i}^2 = -1$. For $z \in \mathbb{C}$, $\Re(z)$ stands for its real part, $\Im(z)$ for its imaginary part, $z^*$ for its complex conjugate and $|z|$ for its modulus.
The set of time steps is denoted by $\timeSet$, the set of buses by $\busSet$, the set of edges by $\mathcal{E}$ (directed towards the root of the tree for acyclic networks), the set of scenarios by $\Omega$.
The impedance of an edge $\overrightarrow{(i,j)} \in \mathcal{E}$ is denoted by $z_{i,j} \in \mathbb{C}$, with resistance $r_{i,j} := \Re(z_{i,j})$, and reactance $x_{i,j} := \Im(z_{i,j})$.
We denote vectors indexed by a set $I$ according to $a:=(a_{i})_{i \in I}$.

\section{The multistage stochastic AC OPF model} \label{sec:opf:model}
\subsection{Formulation of the problem}
Throughout the work, we will make the assumption that the network is radial, connected and passive, i.e., $\Real(z_{i,j}) \geq 0$ and $\Imag(z_{i,j}) \geq 0$ for all lines $\overrightarrow{(i,j)} \in \mathcal{E}$ of the network. We formulate a multi-stage stochastic AC OPF problem with a battery storage system at each bus of the network (except for the reference bus $0$) using the Branch Flow Model{, presented in \cite{low14_1,fari:neal:clar:low:12}, which we extend to the multi-stage stochastic case}. 
We next describe the multistage stochastic AC OPF problem in Branch Flow Model formulation. We impose the following constraints. First, we consider the constraints on voltage squared magnitudes $v$:
\begin{align}
&v_{0,t,\omega} = 1, \quad t \in \timeSet,  \omega \in \Omega, \label{eq:constr:BFM:voltage:slack:multistage}\\
& \underline{v}_i \leq v_{i,t,\omega} \leq \overline{v}_i, \quad  i \in \busSet \setminus \{0\},  t \in \timeSet,  \omega \in \Omega.\label{eq:constr:BFM:voltage:bound:multistage}
\end{align}
We also incorporate bounds on intensity squared magnitude $\mathcal{I}$:
\begin{align}
0 \leq \mathcal{I}_{\overrightarrow{(i,j)},t,\omega} \leq \overline{\mathcal{I}}_{\overrightarrow{(i,j)}}, \quad  \overrightarrow{(i,j)} \in \mathcal{E},  t \in \timeSet,  \omega \in \Omega.\label{eq:constr:BFM:intensity:bound}
\end{align}
We consider bounds on sending-end power flow $S$ magnitudes in the lines of the network, which is a convex quadratic constraint:
\begin{align}
\left|S_{\overrightarrow{(i,j)},t,\omega}\right| \leq \overline{S}_{\overrightarrow{(i,j)}}, \quad  \overrightarrow{(i,j)} \in \mathcal{E},  t \in \timeSet,  \omega \in \Omega.\label{eq:constr:BFM:power:magnitude:bound}
\end{align}
We consider bound constraints on active power injected $p^{\textrm{inj}}$, absorbed $p^{\textrm{abs}}$ by batteries and flexible reactive power injected $q$ {(which can account for SVCs and STATCOMs, \cite{liu:li:wu:orte:17})}:
\begin{align}
& 0 \leq p^{\textrm{inj}}_{i,t,\omega} \leq \overline{p}^{\textrm{inj}}_{i}, \quad  i \in \busSet \setminus \{0\},  t \in \timeSet,  \omega \in \Omega,\label{eq:constr:BFM:power:battery:supply:bound}\\
& 0 \leq p^{\textrm{abs}}_{i,t,\omega} \leq \overline{p}^{\textrm{abs}}_{i}, \quad  i \in \busSet \setminus \{0\}, t \in \timeSet,  \omega \in \Omega,\label{eq:constr:BFM:power:battery:absorb:bound}\\
& \underline{q}_i \leq q_{i,t,\omega} \leq \overline{q}_i, \quad  i \in \busSet \setminus \{0\},   t \in \timeSet,  \omega \in \Omega.\label{eq:constr:BFM:power:battery:reactive:bound}
\end{align}
We introduce the constraints on the states of charge of the batteries $X$, which represent respectively their dynamics (accounting for different efficiencies when charging and discharging the batteries), their initial values and their physical bounds for all $i \in \busSet \setminus \{0\}, \omega \in \Omega${, denoting $\Delta_t$ the time duration of step $t \in \timeSet$}:
\begin{align}
& X_{i,t+1,\omega} = X_{i,t,\omega} + \rho^{\textrm{abs}}_i p^{\textrm{abs}}_{i,t,\omega}\Delta_t  - \rho^{\textrm{inj}}_i p^{\textrm{inj}}_{i,t,\omega} \Delta_t , \  t \in \timeSet \setminus \{T\}, \label{eq:constr:BFM:soc:dynamic}\\
& X_{i,0,\omega} = x_{i} , \label{eq:constr:BFM:soc:init}\\
& \underline{X}_i \leq X_{i,t,\omega} \leq \overline{X}_i , \quad t \in \timeSet.\label{eq:constr:BFM:soc:bounds}
\end{align}
Then we consider the expression of the complex power injections $s$ at all buses of the network $i \in \busSet \setminus \{0\}$, for all time steps $t \in \timeSet$ and scenario $\omega \in \Omega$:
\begin{align}
& s_{i,t,\omega} = p^{\textrm{inj}}_{i,t,\omega} - p^{\textrm{abs}}_{i,t,\omega} + \textbf{i} q_{i,t,\omega} - s^d_{i,t,\omega}.\label{eq:constr:BFM:power:decomp}
\end{align}
This allows to formulate the power balance equations at the non-slack buses and the slack bus $0$, for all $\overrightarrow{(i,j)} \in \mathcal{E},  t \in \timeSet,  \omega \in \Omega$:
\begin{align}
& S_{\overrightarrow{(i,j)},t,\omega} = \sum_{\overrightarrow{(k,i)} \in \mathcal{E}} (S_{\overrightarrow{(k,i)},t,\omega} - z_{k,i} \mathcal{I}_{\overrightarrow{(k,i)},t,\omega}) + s_{i,t,\omega},\label{eq:constr:BFM:power:balance:multistage}\\ 
&  0 = \sum_{\overrightarrow{(k,0)} \in \mathcal{E}} (S_{\overrightarrow{(k,0)},t,\omega} - z_{k,0} \mathcal{I}_{\overrightarrow{(k,0)},t,\omega}) + s_{0,t,\omega}.\label{eq:constr:BFM:power:balance:slack:multistage}
\end{align}
We consider also the voltage propagation constraint and the constraint making the link between voltage squared magnitudes $v$, intensity squared magnitudes $\mathcal{I}$ and sending-end power flows $S$ in the network, for all $\overrightarrow{(i,j)} \in \mathcal{E},  t \in \timeSet,  \omega \in \Omega$:
\begin{align}
& v_{i,t,\omega} - v_{j,t,\omega} = 2 \Real(z_{i,j}^* S_{\overrightarrow{(i,j)},t,\omega}) - |z_{i,j}|^2 \mathcal{I}_{\overrightarrow{(i,j)},t,\omega}, 
\label{eq:constr:BFM:voltage:prop:multistage}\\
& v_{i,t,\omega} \mathcal{I}_{\overrightarrow{(i,j)},t,\omega} = \vert S_{\overrightarrow{(i,j)},t,\omega} \vert^2.\label{eq:constr:BFM:non:convex:multistage}
\end{align}

Constraint \eqref{eq:constr:BFM:non:convex:multistage} is a non-convex quadratic equality constraint.
Last, we consider the non-anticipativity constraint, which encodes the fact that decision variables $y$ should not depend on yet unknown realization of random data of the problem:
\begin{align}
& y \text{ is non anticipative}. \label{eq:constr:BFM:non:anticipativity}
\end{align}
More details shall be given later on the formulation of the non-anticipativity constraints. 
We consider a general (possibly random, progressively-measurable) convex cost function $C$ depending on all decision variables of the problem. We can now formulate the multi-stage stochastic AC-OPF problem:
\begin{align*}
\min_{y=(s_0, s, p^{\textrm{inj}}, p^{\textrm{abs}}, q, X, S,\mathcal{I}, v)} & \Esp{ C(s_0, p^{\textrm{inj}}, p^{\textrm{abs}}, q, S, \mathcal{I}, X)}\\
s.t. \quad &\eqref{eq:constr:BFM:voltage:slack:multistage} - \eqref{eq:constr:BFM:non:anticipativity}.
\end{align*}
We denote this optimization problem by $(P)$. It is a non-convex problem owing to Constraint \eqref{eq:constr:BFM:non:convex:multistage}. {Decision variables are listed and described in Table \ref{tab:decision:var}.}

\begin{table}[hbtp!]
	\centering
	\begin{tabular}{|c|c|}
		\hline
		Decision variable  & Description \\
		\hline
		$s_0 = (s_{0,t,\omega})_{t \in \timeSet,\omega \in\Omega}$ & Apparent power injections at bus $0$\\
		$s = (s_{i,t,\omega})_{i \in \busSet,t \in \timeSet,\omega}$ & Apparent power injections\\
		$p^{\textrm{inj}} = (p^{\textrm{inj}}_{i,t,\omega})_{i \in \busSet,t \in \timeSet,\omega}$ & Injected active power by storage systems\\
		$p^{\textrm{abs}} = (p^{\textrm{abs}}_{i,t,\omega})_{i \in \busSet,t \in \timeSet,\omega}$ & Absorbed active power by storage systems\\
		$q = (q_{i,t,\omega})_{i \in \busSet,t \in \timeSet,\omega}$ & Injected reactive power by flexible systems\\
		$S =(S_{\overrightarrow{(i,j)},t,\omega})_{\overrightarrow{(i,j)} \in \mathcal{E},t \in \timeSet,\omega \in \Omega}$ & Apparent power in lines\\
		$\mathcal{I}=(\mathcal{I}_{\overrightarrow{(i,j)},t,\omega})_{\overrightarrow{(i,j)} \in \mathcal{E},t \in \timeSet,\omega \in \Omega}$ & Squared magnitudes of intensities\\
		$v= (v_{i,t,\omega})_{i \in \busSet,t \in \timeSet,\omega}$ & Squared magnitudes of voltages\\
		\hline
	\end{tabular}
	\caption{Decision variables}
	\label{tab:decision:var}
\end{table}

\subsection{On the formulation of the non-anticipativity constraints}\label{subsec:scenario:tree}

We define a scenario tree, which allows to formulate a problem with a finite number of scenarios while accounting for the filtration structure. Each scenario $\omega \in \Omega$ is associated with a trajectory $\xi_{\omega} = (\xi_{t,\omega})_{t \in \timeSet}$ of the exogenous random process impacting the system and can be visualized as a path from the root to the leaves of the scenario tree. Scenarios $\omega$ and $\omega'$ are said to be indistinguishable up to time $t$ if $\xi_{\tau,\omega} = \xi_{\tau,\omega'}$ for any $\tau \leq t$. If $x = (x_{t,\omega})_{t \in \timeSet,\omega \in \Omega}$ denotes the decision variable of a multi-stage stochastic problem, and $x_{t,\omega}$ denotes the decision taken at time $t$ for scenario $\omega$, non-anticipativity can be expressed by the following constraint for all time $t \in \timeSet$ and all scenarios $\omega$ and $\omega'$ indistinguishable up to time $t$: $ x_{t,\omega} = x_{t,\omega'}$. Good scenario trees should grow exponentially fast with the number of time stages \cite{pflu:pich:14} to appropriately approximate the distribution and the filtration generated by the random noise.


\subsection{Second-Order Cone relaxation of the problem}

As we already observed, constraint \eqref{eq:constr:BFM:non:convex:multistage} is non-convex. Relaxing it into an inequality constraint:
\begin{align}
v_{i,t,\omega}\mathcal{I}_{\overrightarrow{(i,j)},t,\omega} \geq \vert S_{\overrightarrow{(i,j)},t,\omega} \vert^2,\quad  \overrightarrow{(i,j)} \in \mathcal{E}, t \in \mathcal{T}, \omega \in \Omega, \label{eq:constr:BFM:soc:multistage}
\end{align}
yields a convex problem, denoted $(P_{\textrm{SOC}})$. This problem is called the Second-Order Cone Relaxation of the problem, and is given by:
\begin{align*}
\min_{y=(s_0, s, p^{\textrm{inj}}, p^{\textrm{abs}}, q, X, S,\mathcal{I}, v)}\quad& \Esp{ C(s_0, p^{\textrm{inj}}, p^{\textrm{abs}}, q, S, \mathcal{I}, X)}\\
s.t. \quad & \eqref{eq:constr:BFM:voltage:slack:multistage} - \eqref{eq:constr:BFM:voltage:prop:multistage}, \eqref{eq:constr:BFM:non:anticipativity}, \eqref{eq:constr:BFM:soc:multistage}.
\end{align*}

Indeed, \eqref{eq:constr:BFM:soc:multistage} has the structure of a rotated second-order cone constraint $x_1 x_2 \geq x_3^2 + x_4^2$.

\section{Restriction of the feasible set}\label{sec:opf:restrict}

\subsection{Presentation of the problem with restricted feasible set}

We present a variant of Problems $(P)$ and $(P_{\textrm{SOC}})$, obtained by adding a finite number of linear inequalities. Consider additional variables $\vlin = (\vlin_{i,t,\omega})_{i \in \busSet,t \in \timeSet,\omega \in \Omega}$, which plays the role of the square voltage magnitude variable, $\Slin = (\Slin_{\overrightarrow{(i,j)},t,\omega})_{\overrightarrow{(i,j)} \in \mathcal{E},t \in \timeSet,\omega \in \Omega}$, which plays the role of the sending-end power flow variable, and $\slin_0 = (\slin_{0,t,\omega})_{t \in \timeSet,\omega \in \Omega}$, which plays the role of the power injections at the slack bus $0$. 
We first consider the constraints of the Linearized DistFlow model \cite{bara:wu:89:2}:
\begin{align}
&\vlin_{0,t,\omega} = 1, \quad t \in \timeSet,  \omega \in \Omega, \label{eq:constr:BFM:voltage:slack:lin:multistage}\\
& \vlin_{i,t,\omega} \leq \overline{v}_i, \quad  i \in \busSet,   t \in \timeSet,  \omega \in \Omega, \label{eq:constr:BFM:voltage:bound:lin:multistage}\\ 
&\Slin_{\overrightarrow{(i,j)},t,\omega} = \sum_{\overrightarrow{(k,i)} \in \mathcal{E}} \Slin_{\overrightarrow{(k,i)},t,\omega}  + s_{i,t,\omega}, \quad  \overrightarrow{(i,j)} \in \mathcal{E},   t \in \timeSet,  \omega \in \Omega, \label{eq:constr:BFM:power:balance:lin:multistage}\\ 
&  0 = \sum_{\overrightarrow{(k,0)} \in \mathcal{E}} \Slin_{\overrightarrow{(k,0)},t,\omega} + \slin_{0,t,\omega}, \quad  t \in \timeSet,  \omega \in \Omega, \label{eq:constr:BFM:power:balance:slack:lin:multistage}\\
&\vlin_{i,t,\omega} - \vlin_{j,t,\omega} = 2 \Real(z_{i,j}^* \Slin_{\overrightarrow{(i,j)},t,\omega}), \quad  \overrightarrow{(i,j)} \in \mathcal{E},  t \in \timeSet,  \omega \in \Omega.\label{eq:constr:BFM:voltage:prop:lin:multistage}
\end{align}


For all buses $i$, define $\mathcal{E}_i$ as the set of directed edges belonging to the sub-tree starting from $i$, as shown in Figure \ref{fig:subtree:2}. We recall that edges of the network are directed towards the slack bus $0$.
\begin{figure}[h!]
	\centering
	\includegraphics[width = 0.49 \textwidth]{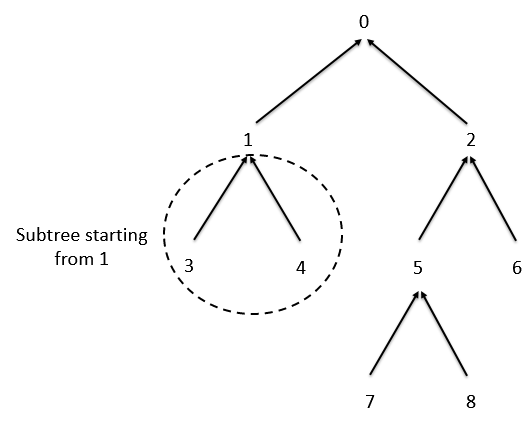}
	\caption{Example of sub-tree $\mathcal{E}_1$}
	\label{fig:subtree:2}
\end{figure}

Consider the additional constraint:
\begin{align}
&\Real(z_{k,l}^* \Slin_{\overrightarrow{(i,j)},t,\omega}) \leq 0 , \quad  \overrightarrow{(i,j)} \in \mathcal{E},  \overrightarrow{(k,l)} \in \mathcal{E}_i,  t \in \timeSet,  \omega \in \Omega.\label{eq:constr:BFM:additional:constr:multistage}
\end{align}
This constraint imposes compensation of active (resp. reactive) reverse power flows in the lines of the network by forward reactive (resp. active) power flow along the same lines. Denoting $y:=(s_0, s, p^{\textrm{inj}}, p^{\textrm{abs}}, q, X, S,\mathcal{I}, v, \Slin,\vlin,\slin_0)$ as the vector of all decision variables, we can now introduce the problem $(P')$:
\begin{align*}
\min_{y}\quad &\Esp{ C(s_0, p^{\textrm{inj}}, p^{\textrm{abs}}, q, S, \mathcal{I}, X)}\\
s.t. \quad & \eqref{eq:constr:BFM:voltage:slack:multistage} - \eqref{eq:constr:BFM:non:anticipativity}, \eqref{eq:constr:BFM:voltage:slack:lin:multistage} - \eqref{eq:constr:BFM:additional:constr:multistage}.
\end{align*}
In particular, the value of $(P')$ is an upper bound on the value of $(P)$: $val(P) \leq val(P')$.

\subsection{Second-order cone relaxation of the problem with restricted feasible set}
Similarly as before, we can introduce the second-order cone relaxation of the new problem $(P')$ by replacing the non-convex quadratic equality constraints \eqref{eq:constr:BFM:non:convex:multistage} by the rotated second-order cone constraints \eqref{eq:constr:BFM:soc:multistage}. This convex relaxation is denoted $(P'_{\textrm{SOC}})$ and can be efficiently solved.
%

\subsection{Conditions ensuring equality of the feasible sets of the original and restricted problems}\label{subsec:opf:equal:restrict:origin}
We show that, under realistic and easily verifiable a priori conditions, the feasible sets of $(P')$ and $(P)$ coincide. Denote $\overline{s}_{i,t,\omega} = \overline{p}^{\textrm{inj}}_i + \textbf{i} \overline{q}_i - s^d_{i,t,\omega}$ an upper bound on total power injections at bus $i$ at time step $t$ for scenario $\omega$, obtained for instance using Constraints \eqref{eq:constr:BFM:power:battery:supply:bound}, \eqref{eq:constr:BFM:power:battery:absorb:bound}, \eqref{eq:constr:BFM:power:battery:reactive:bound} and \eqref{eq:constr:BFM:power:decomp}. For $\omega \in \Omega$, $t \in \timeSet$, define $\vlinbar=(\vlinbar_{t,\omega})_{t \in \timeSet,\omega \in \Omega}$ and $\Slinbar=( \Slinbar_{t,\omega}) _{t \in \timeSet,\omega \in \Omega}$ by:
\begin{align}
\begin{cases}\label{eq:bound:vlin:slin:system}
\Slinbar_{\overrightarrow{(i,j)},t,\omega} = \sum_{\overrightarrow{(k,i)} \in \mathcal{E}} \Slinbar_{\overrightarrow{(k,i)},t,\omega}  + \overline{s}_{i,t,\omega}, \quad  &\overrightarrow{(i,j)} \in \mathcal{E}, \\ 
\vlinbar_{0,t,\omega} = 1, \\
\vlinbar_{i,t,\omega} - \vlinbar_{j,t,\omega} = 2 \Real(z_{i,j}^* \Slinbar_{\overrightarrow{(i,j)},t,\omega}), \quad  &\overrightarrow{(i,j)} \in \mathcal{E}.
\end{cases}
\end{align}

\begin{proposition}\label{prop:suff:cond:equal:p:restricted}
	Define $(\vlinbar,\Slinbar)$ as the unique solution of the system \eqref{eq:bound:vlin:slin:system}. Assume the network is radial, connected and passive (i.e., for all lines $\overrightarrow{(i,j)} \in \mathcal{E}$, it holds that $z_{i,j} \geq_{\mathbb{C}}0$) and moreover:
	\begin{align}\label{eq:suff:cond:equal:p:restricted}
	\begin{cases}
	\vlinbar_{i,t,\omega} \leq \overline{v}_i, \quad  &i \in \busSet,   t \in \timeSet,  \omega \in \Omega,\\
	\Real(z_{k,l}^* \Slinbar_{\overrightarrow{(i,j)},t,\omega}) \leq 0 , \quad  &\overrightarrow{(i,j)} \in \mathcal{E},  \overrightarrow{(k,l)} \in \mathcal{E}_i,  t \in \timeSet,  \omega \in \Omega.
	\end{cases}
	\end{align}
	For any feasible point $y$ of $(P)$ (resp. $(P_{\textrm{SOC}})$), define $(\slin_{0},\vlin,\Slin)$ as the unique solution of the linear system defined by \eqref{eq:constr:BFM:voltage:slack:lin:multistage}-\eqref{eq:constr:BFM:power:balance:lin:multistage}-\eqref{eq:constr:BFM:power:balance:slack:lin:multistage}-\eqref{eq:constr:BFM:voltage:prop:lin:multistage}.
	Then $y' := (y,\slin_0,\vlin,\Slin)$ is feasible with respect to $(P')$ (resp. $(P'_{\textrm{SOC}})$). In particular, $val(P)=val(P')$ and $val(P_{\textrm{SOC}}) = val(P'_{\textrm{SOC}})$.
\end{proposition}

\begin{proof}
	Consider a feasible point $y:=(s,v,S,\mathcal{I})$ of $(P)$ (resp. $(P_{\textrm{SOC}})$). Then, we have $s_{i,t,\omega} \leq_{\mathbb{C}} \overline{s}_{i,t,\omega}$ for all $i \in \busSet \setminus \{0\}$, $t\in \timeSet$, $\omega \in \Omega$. Define $(\slin_{0},\vlin,\Slin)$ by \eqref{eq:constr:BFM:voltage:slack:lin:multistage}-\eqref{eq:constr:BFM:power:balance:lin:multistage}-\eqref{eq:constr:BFM:power:balance:slack:lin:multistage}-\eqref{eq:constr:BFM:voltage:prop:lin:multistage}. In particular, we have:
	\begin{align*}
	\Slin_{\overrightarrow{(i,j)},t,\omega} \leq_{\mathbb{C}} \Slinbar_{\overrightarrow{(i,j)},t,\omega} &\overrightarrow{(i,j)} \in \mathcal{E},   t \in \timeSet,  \omega \in \Omega.
	\end{align*}
	Using the above, the assumption of a passive network and \eqref{eq:suff:cond:equal:p:restricted}, we have:
	\begin{align*}
	\Real(z_{k,l}^* \Slin_{\overrightarrow{(i,j)},t,\omega}) \leq \Real(z_{k,l}^* \Slinbar_{\overrightarrow{(i,j)},t,\omega}) \leq 0, \qquad \overrightarrow{(i,j)} \in \mathcal{E}, \overrightarrow{(k,l)} \in \mathcal{E}_i, t \in \timeSet, \omega \in \Omega,
	\end{align*}
	which shows that \eqref{eq:constr:BFM:additional:constr:multistage} holds.	We also get for all $\overrightarrow{(i,j)} \in \mathcal{E}$, $t \in \timeSet$, $\omega \in \Omega$:
	\begin{align*}
	\vlin_{i,t,\omega} - \vlin_{j,t,\omega} = 2 \Real(z_{i,j}^* \Slin_{\overrightarrow{(i,j)},t,\omega})\leq 2 \Real(z_{i,j}^* \Slinbar_{\overrightarrow{(i,j)},t,\omega})= \vlinbar_{i,t,\omega} - \vlinbar_{j,t,\omega} \enspace,
	\end{align*}
	which implies $\vlin_{i,t,\omega} \leq \vlinbar_{i,t,\omega} \leq \overline{v}_i$ for all $i \in \busSet$, using \eqref{eq:constr:BFM:voltage:slack:multistage}, \eqref{eq:constr:BFM:voltage:slack:lin:multistage} and the orientations of the edges towards the slack bus $0$. This shows that $y':=(y,\slin_0,\vlin,\Slin)$ is feasible for $(P')$ (resp $(P'_{\textrm{SOC}})$).
\end{proof}


\begin{remark}
	The above Proposition implies condition C1 in \cite{huan:16}, which is an abstract assumption on the feasible set of the problem, but it is easier to check.
\end{remark}

The following Proposition shows that if there are no reverse power flows in the network, \eqref{eq:suff:cond:equal:p:restricted} holds.

\begin{proposition}\label{prop:no:reverse:power:flow}
	Define $(\vlinbar,\Slinbar)$ as the unique solution of the system \eqref{eq:bound:vlin:slin:system}. Assume the network is radial, connected and passive, that $\overline{v}_i \geq 1$ for all buses $i \in \busSet$ and the following condition holds:
	\begin{align}\label{eq:suff:cond:equal:p:restricted:2}
	\Slinbar_{\overrightarrow{(i,j)},t,\omega} \leq_{\mathbb{C}} 0, \quad \overrightarrow{(i,j)} \in \mathcal{E},  t \in \timeSet,  \omega \in \Omega.
	\end{align}
	Then \eqref{eq:suff:cond:equal:p:restricted} holds.
\end{proposition}

\begin{proof}
	Under \eqref{eq:suff:cond:equal:p:restricted:2} and the assumption of a passive network, we have:
	$$
	\vlinbar_{i,t,\omega} - \vlinbar_{j,t,\omega} = 2 \Real(z_{i,j}^* \Slinbar_{\overrightarrow{(i,j)},t,\omega}) \leq 0, \qquad \overrightarrow{(i,j)} \in \mathcal{E}, t \in \timeSet, \omega \in \Omega,
	$$
	which implies for all $i \in \busSet$, $t \in \timeSet$, $\omega \in \Omega$, $\vlinbar_{i,t,\omega} \leq \vlinbar_{0,t,\omega} = 1 \leq \overline{v}_i$, using \eqref{eq:constr:BFM:voltage:slack:multistage}, \eqref{eq:constr:BFM:voltage:slack:lin:multistage} and the orientations of the edges towards the slack bus $0$.
	One can then easily show that \eqref{eq:suff:cond:equal:p:restricted} holds.
\end{proof}


\begin{remark}
	Let us notice that condition \eqref{eq:suff:cond:equal:p:restricted:2} is verified if:
	\begin{align}\label{eq:suff:cond:equal:p:restricted:3}
	s_{i,t,\omega} \leq_{\mathbb{C}} 0, \quad i \in \busSet \setminus \{0\},  t \in \timeSet,  \omega \in \Omega.
	\end{align}
\end{remark}

%


\section{Vanishing relaxation gap for the problem with restricted feasible set}\label{sec:opf:np:relax:gap}

We now prove that the problem with restricted feasible set has no relaxation gap, i.e., $val(P'_{\textrm{SOC}}) = val(P')$. 
The proof of this result relies on an appropriate relabeling of the buses, then on an iterative scheme inspired by \cite{huan:16}. 
By comparison with the latter reference, we consider a multi-stage stochastic setting (in particular, we show that non-anticipativity is preserved throughout the iterations) and we allow more general cost functions. We make the following assumption, which can be ensured by appropriately re-indexing the buses:
\begin{enumerate}
	\item[(\bf H.Lab)] The buses are labeled in non-decreasing order according to their depths in the tree, see Figure \ref{fig:subtree:2}.
\end{enumerate}

The iterative scheme we consider takes as input a feasible point $y^{(0)}$ of $(P'_{\textrm{SOC}})$, and at every iteration, constructs a new feasible point of $(P'_{\textrm{SOC}})$ using a Forward-Backward Sweep method, see Algorithm \ref{algo:forward:backward:pass}.
We shall see that the repeated applications of the Forward-Backward Sweep method \ref{algo:forward:backward:pass} generates a convergent sequence of feasible points $(y^{(k)})_{k \in \mathbb{N}}$ of $(P'_{\textrm{SOC}})$, each of them being non-anticipative, and that the limit satisfies the constraints of the non-convex problem $(P')$. 

\begin{algorithm}[h!]
	\caption{Forward-Backward sweep method}
	\label{algo:forward:backward:pass}
	\begin{algorithmic}[1]
		{\small
			\STATE \textbf{Inputs:} $(s_0, S, \mathcal{I}, v)$.
			\FOR{$\omega \in \Omega$, $t \in \mathcal{T}$}
			\FOR{$i = n, n-1, ...,1$}
			\STATE Let $j$ be the unique node in $\mathcal{B}$ such that $\overrightarrow{(i,j)} \in \mathcal{E}$ with the new labels.
			\STATE $S'_{\overrightarrow{(i,j)},t,\omega} \leftarrow s_{i,t,\omega} + \sum_{\overrightarrow{(k,i)} \in \mathcal{E}} (S'_{\overrightarrow{(k,i)},t,\omega} - z_{k,i} \mathcal{I}'_{\overrightarrow{(k,i)},t,\omega})$.
			\STATE $\mathcal{I}'_{\overrightarrow{(i,j)},t,\omega} \leftarrow \frac{\vert S'_{\overrightarrow{(i,j)},t,\omega} \vert^2}{v_{i,t,\omega}}$.
			\ENDFOR
			\STATE $s'_{0,t,\omega} \leftarrow - \sum_{\overrightarrow{(k,0)} \in \mathcal{E}}( S'_{\overrightarrow{(k,0)},t,\omega} - z_{k,0} \mathcal{I}'_{b,0,t,\omega})$.
			\STATE $v'_{0,t,\omega} \leftarrow 1$.
			\FOR{$i = 1, 2, ...,n$}
			\STATE Let $j$ be the unique node in $\mathcal{B}$ such that $\overrightarrow{(i,j)} \in \mathcal{E}$ with the new labels.
			\STATE $v'_{i,t,\omega} \leftarrow v'_{j,t,\omega} + 2 \Real(z_{i,j}^* S'_{\overrightarrow{(i,j)},t,\omega}) - \vert z_{i,j} \vert^2 \mathcal{I}'_{\overrightarrow{(i,j)},t,\omega}$.
			\ENDFOR
			\ENDFOR
			\STATE \textbf{Outputs:} $(s'_0, S', \mathcal{I}', v')$.
		}
	\end{algorithmic}
\end{algorithm}


\begin{lemma}\label{lemma:comparison:ac:linear}
  Let $y:=(s_0, s, p^{\textrm{inj}}, p^{\textrm{abs}}, X, S, \mathcal{I}, v, \vlin, \Slin, \slin_0)$ be a feasible solution of $(P'_{\textrm{SOC}})$, with . Then, if the network is passive (meaning that for all lines $\overrightarrow{(i,j)} \in \mathcal{E}$, we have $z_{i,j} \geq_{\mathbb{C}}0$), the following inequalities
  are valid:
	\begin{align*}
	&S_{\overrightarrow{(i,j)},t,\omega} \leq_{\mathbb{C}} \Slin_{\overrightarrow{(i,j)},t,\omega}, \quad \forall \overrightarrow{(i,j)} \in \mathcal{E}, t \in \timeSet, \omega \in \Omega,\\
	&s_{0,t,\omega} \geq _{\mathbb{C}} \slin_{0,t,\omega}, \quad \forall t \in \timeSet, \omega \in \Omega,\\
	&v_{i,t,\omega} \leq \vlin_{i,t,\omega}, \quad \forall i \in \busSet, t \in \timeSet, \omega \in \Omega.
	\end{align*}
\end{lemma}

\begin{proof}
	The claimed inequalities are established $t$ by $t$ and $\omega$ by $\omega$. We drop the corresponding indices for simplicity of the notations. The inequalities on $S$ and $\Slin$ arise from the constraint \eqref{eq:constr:BFM:intensity:bound} which implies that $\mathcal{I}$ is non-negative component-wise, from passivity of the network and from constraints \eqref{eq:constr:BFM:power:balance:multistage} and \eqref{eq:constr:BFM:power:balance:lin:multistage}. The inequalities on $s_0$ and $\slin_0$ can then be deduced by the inequality between $S$ and $\Slin$ and constraints \eqref{eq:constr:BFM:power:balance:slack:multistage} and \eqref{eq:constr:BFM:power:balance:slack:lin:multistage}. 
	Comparing \eqref{eq:constr:BFM:voltage:prop:multistage} and \eqref{eq:constr:BFM:voltage:prop:lin:multistage}, using the passivity of the network and the inequalities between $S$ and $\Slin$, one gets for all $\overrightarrow{(i,j)}$ in $\mathcal{E}$:
	$$	v_{i} - v_{j} \leq \vlin_{i} - \vlin_{j}.$$
	We can then show the inequalities on $v$ and $\vlin$ using the fact that $\vlin_{0}=1=v_{0}$, by \eqref{eq:constr:BFM:voltage:slack:multistage} and \eqref{eq:constr:BFM:voltage:slack:lin:multistage}, and using the fact that edges are directed towards the slack bus (root of the tree) indexed by $0$.
\end{proof}

\begin{lemma}\label{lemma:monotone:algo:opf}
	Algorithm \ref{algo:forward:backward:pass} is well-posed. Let $y$ be a feasible solution of $(P'_{\textrm{SOC}})$, defined by $y:= (s_0, s, p^{\textrm{inj}}, p^{\textrm{abs}}, X, S, \mathcal{I}, v, \vlin, \Slin, \slin_0)$. Apply Algorithm \ref{algo:forward:backward:pass} once to $(s_0,S,\mathcal{I},v)$ and denote by $(s'_0 ,S' ,\mathcal{I}', v')$ its output.
	Then we have:
	\begin{align*}
	&S_{\overrightarrow{(i,j)},t,\omega} \leq_{\mathbb{C}} S'_{\overrightarrow{(i,j)},t,\omega} , \quad \forall \overrightarrow{(i,j)} \in \mathcal{E}, t \in \timeSet, \omega \in \Omega,\\
	& \left| S_{\overrightarrow{(i,j)},t,\omega} \right| \geq \left| S'_{\overrightarrow{(i,j)},t,\omega} \right| , \quad \forall \overrightarrow{(i,j)} \in \mathcal{E}, t \in \timeSet, \omega \in \Omega,\\
	&\mathcal{I}_{\overrightarrow{(i,j)},t,\omega} \geq \mathcal{I}'_{\overrightarrow{(i,j)},t,\omega} , \quad \forall \overrightarrow{(i,j)} \in \mathcal{E}, t \in \timeSet, \omega \in \Omega,\\
	&s_{0,t,\omega} \geq _{\mathbb{C}} s'_{0,t,\omega}, \quad \forall t \in \timeSet, \omega \in \Omega,\\
	&v_{i,t,\omega} \leq v'_{i,t,\omega}, \quad \forall i \in \busSet, t \in \timeSet, \omega \in \Omega.
	\end{align*}
	Moreover, $y':=(s_0, s, p^{\textrm{inj}}, p^{\textrm{abs}}, X, S, \mathcal{I}, v, \vlin, \Slin, \slin_0)$ is feasible for $(P'_{\textrm{SOC}})$. In particular, it is non-anticipative.
\end{lemma}

\begin{proof}
	The claimed inequalities are established $t$ by $t$ and $\omega$ by $\omega$. We drop the corresponding indices for simplicity of the notations.
	The definition of $S$ for leaves of the tree in the forward pass is well-defined as the sum in the LHS is empty in this case by our labels (Algorithm \ref{algo:forward:backward:pass}, line 6). 
	The labels chosen ensure that the forward pass always explores leaves before their ancestors, which ensures that the forward pass is well-defined. Therefore, the whole algorithm is well-posed.
	Consider the forward pass, with $i=n$, $n$ being the index of the last bus after setting the new labels (see Assumption \textbf{(H.Lab)}). Denoting $j$ its unique ancestor, we have $S_{\overrightarrow{(n,j)}} = s_{n} = S'_{\overrightarrow{(n,j)}}$ by construction. 
	We then obtain, using the fact that $y$ satisfies \eqref{eq:constr:BFM:soc:multistage}:
	$$\mathcal{I}'_{\overrightarrow{(n,j)}} = |S'_{\overrightarrow{(n,j)}}|^2/v_{n} = |S_{\overrightarrow{(n,j)}}|^2/v_{n} \leq  \mathcal{I}_{\overrightarrow{(n,j)}}.$$
	
	Let us now assume $i<n$, and we assume the inequalities for $S$, $S'$, $\mathcal{I}$ and $\mathcal{I}'$ have been proved for all $k = i+1,...,n$. 
	If $i$ is a leaf, we can prove the inequalities similarly as for $i=n$. 
	Consider the case where $i<n$ is not a leaf. Let $j$ be its unique ancestor. Then, by passivity of the network:
	$$
	S'_{\overrightarrow{(i,j)}} = s_{i} + \sum_{\overrightarrow{(k,i)} \in \mathcal{E}} (S'_{\overrightarrow{(k,i)}} - z_{k,i} \mathcal{I}'_{\overrightarrow{(k,i)}})\geq_{\mathbb{C}} s_{i} + \sum_{\overrightarrow{(k,i)} \in \mathcal{E}} (S_{\overrightarrow{(k,i)}} - z_{k,i} \mathcal{I}_{\overrightarrow{(k,i)}}) = S_{\overrightarrow{(i,j)}}.
	$$
	Besides, denoting $P:= \Real(S)$, $Q :=\Imag(S)$, $\Plin:= \Real(\Slin)$ and $\Qlin:= \Imag(\Slin)$:
	\begin{align*}
	\vert S'_{\overrightarrow{(i,j)}} \vert^2 -  \vert S_{\overrightarrow{(i,j)}} \vert^2 & = (P'_{\overrightarrow{(i,j)}}+ P_{\overrightarrow{(i,j)}})(P'_{\overrightarrow{(i,j)}} - P_{\overrightarrow{(i,j)}})+ (Q'_{\overrightarrow{(i,j)}} + Q_{\overrightarrow{(i,j)}})(Q'_{\overrightarrow{(i,j)}} - Q_{\overrightarrow{(i,j)}})\\
	&\leq 2 \Plin_{\overrightarrow{(i,j)}}(P'_{\overrightarrow{(i,j)}} - P_{\overrightarrow{(i,j)}})  + 2 \Qlin_{\overrightarrow{(i,j)}} (Q'_{\overrightarrow{(i,j)}} - Q_{\overrightarrow{(i,j)}}) \\
	&= -2\left(\sum_{(k,l) \in \mathcal{E}_i} (\Plin_{\overrightarrow{(i,j)}}r_{k,l} + \Qlin_{\overrightarrow{(i,j)}}x_{k,l} )(\mathcal{I}'_{\overrightarrow{(k,l)}} - \mathcal{I}_{\overrightarrow{(k,l)}})\right)\\
	&= -2 \left(\sum_{\overrightarrow{(k,l)} \in \mathcal{E}_i} \Real(z_{k,l}^* \Slin_{\overrightarrow{(i,j)}})(\mathcal{I}'_{\overrightarrow{(k,l)}} - \mathcal{I}_{\overrightarrow{(k,l)}})\right)\\
	& \leq 0.
	\end{align*}
	In the inequality in the third line, we used Lemma \ref{lemma:comparison:ac:linear}, then we use the definition of $S'$ and the fact that $S$ satisfies \eqref{eq:constr:BFM:power:balance:multistage} to obtain the following equality. The last inequality is obtained using $\mathcal{I}'_{\overrightarrow{(k,l)}} \geq \mathcal{I}_{\overrightarrow{(k,l)}}$ for all $\overrightarrow{(k,l)} \in \mathcal{E}_i$ and the fact that $\Slin$ satisfies \eqref{eq:constr:BFM:additional:constr:multistage}. 
	We then obtain:
	$$\mathcal{I}'_{\overrightarrow{(i,j)}} = |S'_{\overrightarrow{(i,j)}}|^2/v_{i} \leq |S_{\overrightarrow{(i,j)}}|^2/v_{i} \leq  \mathcal{I}_{\overrightarrow{(i,j)}}.$$
	
	The inequality $s_{0} \geq _{\mathbb{C}} s'_{0}$ can then be deduced from the first and third above inequality and the assumption of passivity of the network.
	We have $v_0 = 1 =v'_0$ by construction and by \eqref{eq:constr:BFM:voltage:slack:multistage}.
	Notice then that for all $\overrightarrow{(i,j)} \in \mathcal{E}$, by construction of $v'$ and by \eqref{eq:constr:BFM:voltage:prop:multistage}:
	$$
	v'_i - v_i \geq v'_j - v_j,
	$$
	where we used the earlier inequalities on $S$, $S'$, $\mathcal{I}$ and $\mathcal{I}'$ and the assumption of passivity of the network. By propagating this in the network from $0$ to $n$, we get the desired inequality on the voltage squared magnitudes. 
	We have:
	\begin{align*}
	\forall \overrightarrow{(i,j)} \in \mathcal{E}, \quad \mathcal{I}'_{\overrightarrow{(i,j)}} = \vert S'_{\overrightarrow{(i,j)}}\vert^2/v_i \geq \vert S'_{\overrightarrow{(i,j)}}\vert^2/v'_i.
	\end{align*}
	Hence $(S',v',\mathcal{I}')$ satisfies \eqref{eq:constr:BFM:soc:multistage}. Let $y':=(s_0, s, p^{\textrm{inj}}, p^{\textrm{abs}}, X, S, \mathcal{I}, v, \vlin, \Slin, \slin_0)$ be the new point. Non-anticipativity of $y'$ arises from the fact that for all $t\in \timeSet$ and $\omega \in \Omega$, $y'_{t,\omega}$ is measurable with respect to $y_{t,\omega}$. By construction and using the inequalities derived above as well as Lemma \ref{lemma:comparison:ac:linear}, one can show that $y'$ is feasible for $(P'_{\textrm{SOC}})$ if $y$ is feasible. 
\end{proof}

\begin{corollary}\label{corollary:opf:rounding}
	Let $y:=(s_0, s, p^{\textrm{inj}}, p^{\textrm{abs}}, X, S, \mathcal{I}, v, \vlin, \Slin, \slin_0)$ be a feasible solution of $(P'_{\textrm{SOC}})$. Then, there exists a feasible (non-anticipative) point for $(P')$, denoted by $y':= (s'_0, s, p^{\textrm{inj}}, p^{\textrm{abs}}, X, S', \mathcal{I}', v', \vlin, \Slin, \slin_0)$ such that:
	\begin{align*}
	&S_{\overrightarrow{(i,j)},t,\omega} \leq_{\mathbb{C}} S'_{\overrightarrow{(i,j)},t,\omega} , \quad \forall \overrightarrow{(i,j)} \in \mathcal{E}, t \in \timeSet, \omega \in \Omega,\\
	& \left| S_{\overrightarrow{(i,j)},t,\omega} \right| \geq \left| S'_{\overrightarrow{(i,j)},t,\omega} \right| , \quad \forall \overrightarrow{(i,j)} \in \mathcal{E}, t \in \timeSet, \omega \in \Omega,\\
	&\mathcal{I}_{\overrightarrow{(i,j)},t,\omega} \geq \mathcal{I}'_{\overrightarrow{(i,j)},t,\omega} , \quad \forall \overrightarrow{(i,j)} \in \mathcal{E}, t \in \timeSet, \omega \in \Omega,\\
	&s_{0,t,\omega} \geq _{\mathbb{C}} s'_{0,t,\omega}, \quad \forall t \in \timeSet, \omega \in \Omega,\\
	&v_{i,t,\omega} \leq v'_{i,t,\omega}, \quad \forall i \in \busSet, t \in \timeSet, \omega \in \Omega.
	\end{align*}
\end{corollary}

\begin{proof}
	The claimed inequalities are established $t$ by $t$ and $\omega$ by $\omega$. We drop the corresponding indices for simplicity of the notations.
	Apply inductively Algorithm \ref{algo:forward:backward:pass} to $x^{(0)}:=(s_0, S, \mathcal{I},v)$. This defines a sequence $(x^{(k)})_{k \in \mathbb{N}}:=(s^{(k)}_0, S^{(k)}, \mathcal{I}^{(k)},v^{(k)})_{k \in \mathbb{N}}$. By Lemmas \ref{lemma:comparison:ac:linear} and \ref{lemma:monotone:algo:opf}, we have the following inequalities for all $k \in \mathbb{N}$:
	\begin{align*}
	&S^{(k)}_{\overrightarrow{(i,j)}} \leq_{\mathbb{C}} S^{(k+1)}_{\overrightarrow{(i,j)}} \leq_{\mathbb{C}} \Slin_{\overrightarrow{(i,j)}} , \quad \forall \overrightarrow{(i,j)} \in \mathcal{E},\\
	&\mathcal{I}^{(k)}_{\overrightarrow{(i,j)}} \geq \mathcal{I}^{(k+1)}_{\overrightarrow{(i,j)}} \geq 0 , \quad \forall \overrightarrow{(i,j)} \in \mathcal{E},\\
	&s^{(k)}_{0} \geq _{\mathbb{C}} s^{(k+1)}_{0}\geq _{\mathbb{C}} \slin_{0} ,\\
	&v^{(k)}_{i} \leq v^{(k+1)}_{i} \leq \vlin_i, \quad \forall i \in \busSet.
	\end{align*}
	By the monotone convergence theorem, $(x^{(k)})$ converges to a point $(s'_0,S',\mathcal{I}',v')$. Define	$y':= (s'_0, s, p^{\textrm{inj}}, p^{\textrm{abs}}, X, S', \mathcal{I}', v', \vlin, \Slin, \slin_0),$
	which is feasible for $(P'_{\textrm{SOC}})$ as a limit of feasible points of $(P'_{\textrm{SOC}})$. Besides, it satisfies \eqref{eq:constr:BFM:non:convex:multistage} as $(s'_0,S',\mathcal{I}',v')$ is a fixed point of Algorithm \ref{algo:forward:backward:pass}. This implies that $y'$ is a feasible point of $(P')$. The inequalities claimed arise from the monotone behavior of the sequence $x^{(k)}$.
\end{proof}

\begin{theorem}\label{theorem:opf:zero:duality:gap:restricted}
	Assume the following:
	\begin{enumerate}
		\item The network is radial and connected.
		\item The network is passive, i.e., $\Real(z_{i,j}) \geq 0$ and $\Imag(z_{i,j}) \geq 0$ for all lines $\overrightarrow{(i,j)} \in \mathcal{E}$ of the network.
		\item The cost function $C$ is convex, component-wise monotone non-increasing in $\Real(S)$, $\Imag(S)$ and $v$, component-wise monotone non-decreasing in $\mathcal{I}$, $\Real(s_0)$, $\Imag(s_0)$ and $\vert S \vert$.
	\end{enumerate}
	Then $(P')$ has no relaxation gap, i.e., its optimal value coincides with the optimal value of $(P'_{\textrm{SOC}})$.
\end{theorem}
\begin{proof}
	If $(P'_{\textrm{SOC}})$ is infeasible, then so is $(P')$ and the result holds.
	If $(P'_{\textrm{SOC}})$ is feasible and bounded from below, consider its optimal solution $y^*$. Corollary \ref{corollary:opf:rounding} and the monotonicity assumptions on the cost then show that there exists $\tilde{y}^*$ which is feasible for $(P')$ and with lower cost than $y^*$. This yields the result.
	If $(P'_{\textrm{SOC}})$ is feasible and unbounded from below, given a sequence of feasible points of $(P'_{\textrm{SOC}})$, whose costs goes to $-\infty$, we build a sequence of feasible point of $(P')$ with lower costs, using Corollary \ref{corollary:opf:rounding}. This shows that $(P')$ is also feasible and unbounded from below.
\end{proof}

The assumptions of Theorem \ref{theorem:opf:zero:duality:gap:restricted} are quite realistic. In practice, the cost functional is often independent from $\Real(S), \Imag(S), v, \Imag(s_0)$ and is monotone non-decreasing in the active power injections at the substation $\Real(s_0)$ and in thermal losses, proportional to $\mathcal{I}$. Besides, no assumption is made on the behavior of the cost regarding power injections at buses other than the slack bus. This allows to apply the result to a wide range of applications.

The following theorem is an immediate consequence of Theorem \ref{theorem:opf:zero:duality:gap:restricted}.

\begin{theorem}[A posteriori bound on the relaxation gap]\label{theorem:a:posteriori:bound:relaxation:gap}
	Under the assumptions of Theorem \ref{theorem:opf:zero:duality:gap:restricted}, the relaxation gap of $(P)$, given by $val(P) - val(P_{\textrm{SOC}})$ is bounded from above by $val(P'_{\textrm{SOC}}) - val(P_{\textrm{SOC}})$.\hfill \qed
\end{theorem}

The following theorem is a consequence of Proposition \ref{prop:suff:cond:equal:p:restricted} and Theorem \ref{theorem:opf:zero:duality:gap:restricted}.

\begin{theorem}[A priori condition for a vanishing relaxation gap]\label{theorem:opf:zero:duality:gap}
	Under the assumptions of Proposition \ref{prop:suff:cond:equal:p:restricted} and Theorem \ref{theorem:opf:zero:duality:gap:restricted}, the problem $(P)$ has no relaxation gap, i.e., $val(P) = val(P_{\textrm{SOC}})$.\hfill \qed
\end{theorem}


\begin{remark}\label{remark:generalization:storage:constraints}
	The result can be generalized to other types of electricity storage systems, thermal storage systems, electrical vehicles\dots 
	Other constraints, static or dynamic (i.e., linking variables of two distinct time steps), can be incorporated to power injections at all buses except the slack bus $0$.
\end{remark}

\section{Discussion on applicability to real-world distribution networks}\label{sec:discussion:real:networks}

\subsection{Validity of the results for uncontrollable Voltage Regulation Transformers (VRTs)}

{For flexibility of the model, we can also incorporate an ideal transformer for each bus of the network, which tap ratio is given by $t_i \in \mathbb{R}^*_+$ ($t_i=1$ if no such transformer is present at bus $i$). Doing so allows for instance to model VRTs with fixed configuration \cite{liu:li:wu:orte:17}. To this end, for each bus $i$, we introduce an additional virtual bus $i-$. Buses $i$ and $i-$ are the two ends of the transformer, $i$ being closer to the root of the network by convention, see Figure \ref{fig:transformer}.}

\begin{figure}[h!]
	\centering
	\includegraphics[width=0.7 \textwidth]{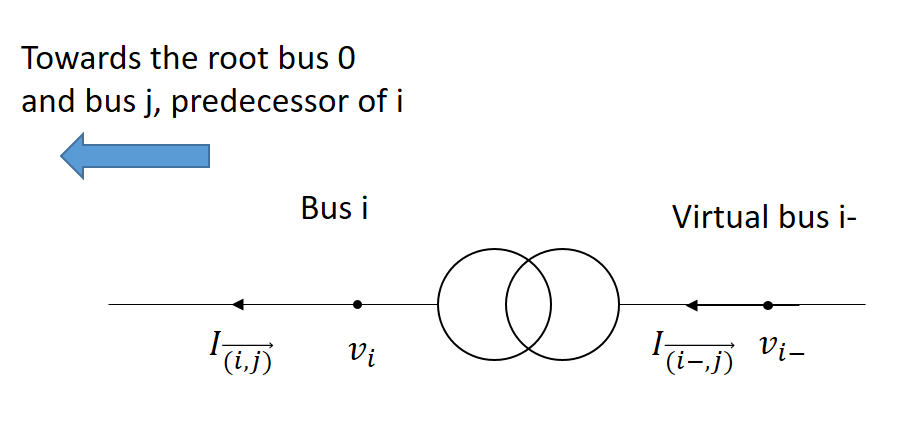}
	\caption{Transformer at bus $i$}\label{fig:transformer}
\end{figure}

{The voltage square magnitude at the virtual bus $i-$ is then given by $|t_i|^2 v_i$. In particular, ideal transformers only change voltage magnitude, but not voltage phasis, nor apparent power. In problems $(P)$ and $(P_{SOC})$, only \eqref{eq:constr:BFM:voltage:prop:multistage} is replaced by:
\begin{align}\label{eq:constr:BFM:voltage:prop:multistage:transfo}
	v_{i,t,\omega} - |t_j|^2 v_{j,t,\omega} = 2 \Real(z_{i,j}^* S_{\overrightarrow{(i,j)},t,\omega}) - |z_{i,j}|^2 \mathcal{I}_{\overrightarrow{(i,j)},t,\omega}.
\end{align}
In case the transformers are not controllable, i.e., the tap ratios are fixed, then the equation remains linear, and the problem remains of the same nature. In problems $(P')$ and $(P'_{SOC})$, additionally to the previous modification, \eqref{eq:constr:BFM:voltage:prop:lin:multistage} is replaced by:
\begin{align}\label{eq:constr:BFM:voltage:prop:lin:multistage:transfo}
\vlin_{i,t,\omega} - |t_j|^2\vlin_{j,t,\omega} = 2 \Real(z_{i,j}^* \Slin_{\overrightarrow{(i,j)},t,\omega}), \quad  \overrightarrow{(i,j)} \in \mathcal{E},  t \in \timeSet,  \omega \in \Omega.
\end{align}
This equation remains linear if the tap ratios of the transformers are fixed. Moreover, the last equation in \eqref{eq:bound:vlin:slin:system} is replaced by:
\begin{align*}
\vlinbar_{i,t,\omega} - |t_j|^2 \vlinbar_{j,t,\omega} = 2 \Real(z_{i,j}^* \Slinbar_{\overrightarrow{(i,j)},t,\omega}), \quad  &\overrightarrow{(i,j)} \in \mathcal{E}.
\end{align*}
Propositions \ref{prop:suff:cond:equal:p:restricted} and \ref{prop:no:reverse:power:flow} remain valid in the setting of uncontrollable transformers, by small adaptations of the proofs. Line 12 of Algorithm \ref{algo:forward:backward:pass} is modified as:
\begin{align*}
v'_{i,t,\omega} \leftarrow |t_j|^2 v'_{j,t,\omega} + 2 \Real(z_{i,j}^* S'_{\overrightarrow{(i,j)},t,\omega}) - \vert z_{i,j} \vert^2 \mathcal{I}'_{\overrightarrow{(i,j)},t,\omega}.
\end{align*}
Lemmas \ref{lemma:comparison:ac:linear}, \ref{lemma:monotone:algo:opf}, Corollary \ref{corollary:opf:rounding}, Theorems \ref{theorem:opf:zero:duality:gap:restricted}, \ref{theorem:a:posteriori:bound:relaxation:gap} and \ref{theorem:opf:zero:duality:gap} remain valid in the case of uncontrollable transformers.
}

\subsection{Extension to controllable VRT}
{Considering tap ratios of transformers as decision variables introduce non-linearity (and non-convexity) in Constraints \eqref{eq:constr:BFM:voltage:prop:multistage:transfo} and \eqref{eq:constr:BFM:voltage:prop:lin:multistage:transfo}. Extension of our results to this case is an interesting perspective of our work left for future research. Let us provide some intuition on a possible modeling procedure though, inspired by \cite{liu:li:wu:orte:17,liu:li:wu:18}. If the tap ratios take value in a continuous interval $[t_i^{\tt min}, t_i^{\tt max}] \subset (0, + \infty)$, then $(P)$ (resp. $(P_{SOC})$) is unchanged up to introduction to additional decision variables $w_{\overrightarrow{(i,j)},t,\omega}$ for all $\overrightarrow{(i,j)} \in \mathcal{E}$, $t \in \timeSet$, $\omega \in \Omega$ and Constraint \eqref{eq:constr:BFM:voltage:prop:multistage} is replaced by:
\begin{align*}
	&v_{i,t,\omega} - w_{j,t,\omega} = 2 \Real(z_{i,j}^* S_{\overrightarrow{(i,j)},t,\omega}) - |z_{i,j}|^2 \mathcal{I}_{\overrightarrow{(i,j)},t,\omega}, \quad & \overrightarrow{(i,j)} \in \mathcal{E}, t \in \timeSet, \omega \in \Omega,\\
	&|t^{\tt min}_i|^2 v_{i,t,\omega} \leq w_{i,t,\omega}\leq |t^{\tt max}_i|^2 v_{i,t,\omega}, \quad & i \in \mathcal{B}, t \in \timeSet, \omega \in \Omega.
\end{align*}
For $(P')$ and $(P'_{SOC})$, one needs other additional decision variables $w^{\textrm{Lin}}_{\overrightarrow{(i,j)},t,\omega}$ and Constraint \eqref{eq:constr:BFM:voltage:prop:lin:multistage} is replaced by:
\begin{align*}
&\vlin_{i,t,\omega} - w^{\textrm{Lin}}_{i,t,\omega} = 2 \Real(z_{i,j}^* \Slin_{\overrightarrow{(i,j)},t,\omega}), \quad & \overrightarrow{(i,j)} \in \mathcal{E}, t \in \timeSet, \omega \in \Omega,\\
&|t^{\tt min}_i|^2 \vlin_{i,t,\omega} \leq w^{\textrm{Lin}}_{i,t,\omega}\leq |t^{\tt max}_i|^2 \vlin_{i,t,\omega}, \quad & i \in \mathcal{B}, t \in \timeSet, \omega \in \Omega.
\end{align*}
Extension of our results to this setting require an adaptation of step 12 of Algorithm \ref{algo:forward:backward:pass} in order to give an explicit updating rule for $w$ and $v$.
In the case where the tap ratios of transformers are decision variables with value in (finite) discrete sets, it is possible to model $(P)$ and $(P')$ as Mixed-Integer Non-Linear Programming problems and $(P_{SOC})$ and $(P'_{SOC})$ and Mixed-Integer Second Order Cone Programming problems, see \cite{liu:li:wu:orte:17, liu:li:wu:18}. Our results can be applied in at the leaves of a branch-and-bound tree, when all tap ratios of VRTs are fixed.

\subsection{Other voltage regulation devices}
{Introducing capacitor banks in the model requires the modeling of nodal shunt elements, which requires in turn a modification of the Branch Flow Model \cite{fari:13}. In particular, \eqref{eq:constr:BFM:power:balance:multistage} is replaced by:
\begin{align}
& S_{\overrightarrow{(i,j)},t,\omega} = \sum_{\overrightarrow{(k,i)} \in \mathcal{E}} (S_{\overrightarrow{(k,i)},t,\omega} - z_{k,i} \mathcal{I}_{\overrightarrow{(k,i)},t,\omega}) + y_i^{\tt sh} v_{i,t,\omega} + s_{i,t,\omega}.\label{eq:constr:BFM:power:balance:multistage:shunt}
\end{align}
Similarly, \eqref{eq:constr:BFM:power:balance:lin:multistage} is replaced by:
\begin{align}
& \Slin_{\overrightarrow{(i,j)},t,\omega} = \sum_{\overrightarrow{(k,i)} \in \mathcal{E}} \Slin_{\overrightarrow{(k,i)},t,\omega} + y_i^{\tt sh} \vlin_{i,t,\omega} + s_{i,t,\omega}.\label{eq:constr:BFM:power:balance:lin:multistage:shunt}
\end{align}
Due to the additional dependency of $S$ and $\Slin$ on $v$ and $\vlin$ respectively, it remains unclear whether the key Lemmas \ref{lemma:comparison:ac:linear} and \ref{lemma:monotone:algo:opf} remain valid, even under some sign conditions on the real and imaginary part of their associated admittance. This is due to the fact that the proofs of these lemmas relied on the tree structure of the network and on \eqref{eq:constr:BFM:power:balance:multistage} and \eqref{eq:constr:BFM:power:balance:lin:multistage} being independent from $v$, which allowed to directly prove comparison relation for apparent power. Therefore, whether our results are valid or not with capacitor banks remains an open question. 
On the other hand, SVCs and STATCOMs can inject or absorb reactive power at the buses of the network and our model already allows this possibility. Such devices can thus be incorporated in the model while preserving our results.

\subsection{Extension to unbalanced multi-phase networks}
{In the multi-phase setting, the natural convex relaxation of the problem becomes a SDP relaxation \cite{liu:li:wu:orte:17, liu:li:wu:18,gan:low:14} and the non-convex equality constraints \eqref{eq:constr:BFM:non:convex:multistage} are replaced by rank-one constraints for 6 x 6 matrices for three-phase networks. Thus, we do not expect our fixed point procedure (consisting in repeating an adaptation of Algorithm \ref{algo:forward:backward:pass} until convergence) to yield solutions satisfying such rank conditions. Or at least, we do not expect to be able to prove it easily. This is left for further research.}

\section{Case studies and numerical illustration} \label{sec:opf:numerics}

This numerical study is implemented using Matlab R2018b combined
with YALMIP R20200116, interfaced with conic solver Gurobi 9.0.0 with an Intel-Core i7 PC at 2.1 GHz with 16 Go memory.

\subsection{Network topology}
We consider a distribution network on the Southern California Edison system with 56 buses \cite{fari:neal:clar:low:12}. A visualization of the network before relabeling the buses is given in Figure \ref{fig:network:viz:56}.

\begin{figure}[!h]
	\centering{\includegraphics[trim=2cm 2cm 2cm 1cm,clip,width=0.5\textwidth]{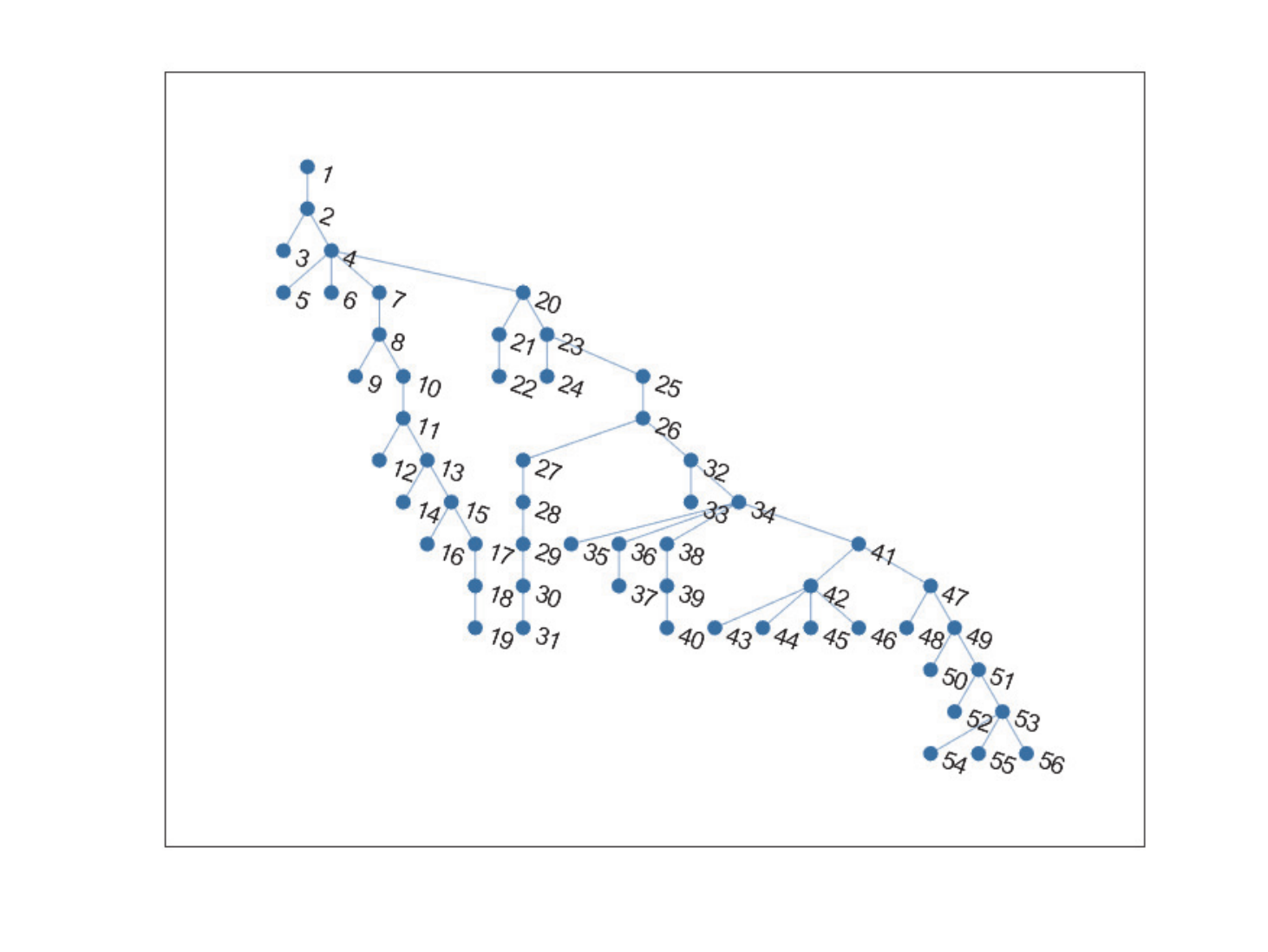}}
	\caption{56 buses network \cite{fari:neal:clar:low:12}\label{fig:network:viz:56}}
\end{figure}

We leave network topology and line resistances and reactances unchanged compared with \cite{fari:neal:clar:low:12}. The peak loads are assumed unchanged, but we modify the installed solar capacities and do not assume shunt capacitors at any bus. Additionally, we shall considered storage systems. We attribute to each bus $i$ a size parameter $S_i$, given by the peak load of the bus $i$, taken from the column Load data, Peak MVA. Non-specified buses are attributed the size $S_i = 0$. The values of the parameters of the network are given in Figure \ref{tab:network:param:56}. 
In particular, the total peak load is given by $\sum_i S_i = 3.835$ MVA.

\begin{figure*}[h!]
	\centering
	\includegraphics[width=\textwidth]{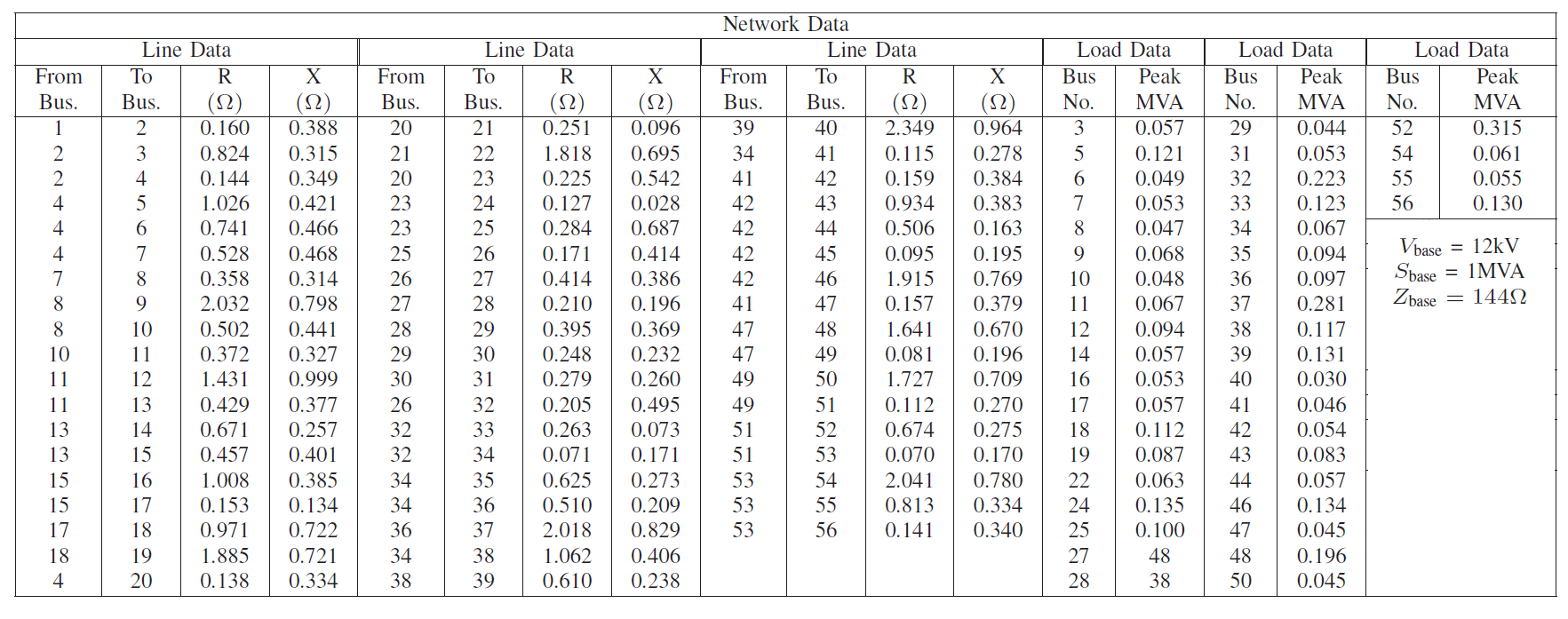}
	\caption{Table of the data of the 56 buses network \cite{fari:neal:clar:low:12}: line resistances, reactances, peak loads $S_i$}\label{tab:network:param:56}
\end{figure*}


We assume the maximal intensity magnitude in each line is $300$ A, which yields the bound $\overline{\mathcal{I}}_{\overrightarrow{(i,j)}} = 90000$ A$^2$. We accept $5 \%$ deviations of voltage magnitude from the reference value, i.e., $\underline{v}_{\overrightarrow{(i,j)}} = (0.95)^2$ p.u. and $\overline{v}_{\overrightarrow{(i,j)}} = (1.05)^2$ p.u. We assume also $\overline{S}_{\overrightarrow{(i,j)}} = 5$ MVA.



\subsection{Discussion about other electronics devices}

{
	More generally, incorporating devices which only modify the bounds on power injections at buses except the reference (like storage systems, flexible consumption, energy conversion systems...) do not make Theorem \ref{theorem:opf:zero:duality:gap:restricted} invalid. Besides, a priori conditions guaranteeing a vanishing relaxation gap (see Theorem \ref{theorem:opf:zero:duality:gap}) should be verified with the updated bounds on power injections. For devices which directly impact voltage or intensity magnitude or link them with power injections like shunt capacitors (see the model of \cite{fari:neal:clar:low:12}) or transformers, one should extend the analysis developed above to check if similar results may hold under further assumptions.}


\subsection{Exogenous residual demand}
We consider an exogenous demand profile at node $i$ and for time $t$ given by the difference between a deterministic consumption profile and a solar power production profile
$$s^{d}_{i,t,\omega} = s^{\textrm{cons}}_{i,t} - p^{\textrm{sol}}_{i,t,\omega}.$$
The consumption profiles are given by $s^{\textrm{cons}}_{i,t} = s^{\textrm{cons,ref}}_{t} \frac{1 + \textbf{i} 0.2}{\sqrt{1+ (0.2)^2}} S_i$, i.e., the consumption at each node $i$ is proportional to a reference deterministic evolution $s^{\textrm{cons,ref}}$ (independent from $i$) represented in Figure \ref{fig:conso} and the size parameter $S_i$.

\begin{figure}[h!]
	\centering
	\includegraphics[width=0.5\textwidth] {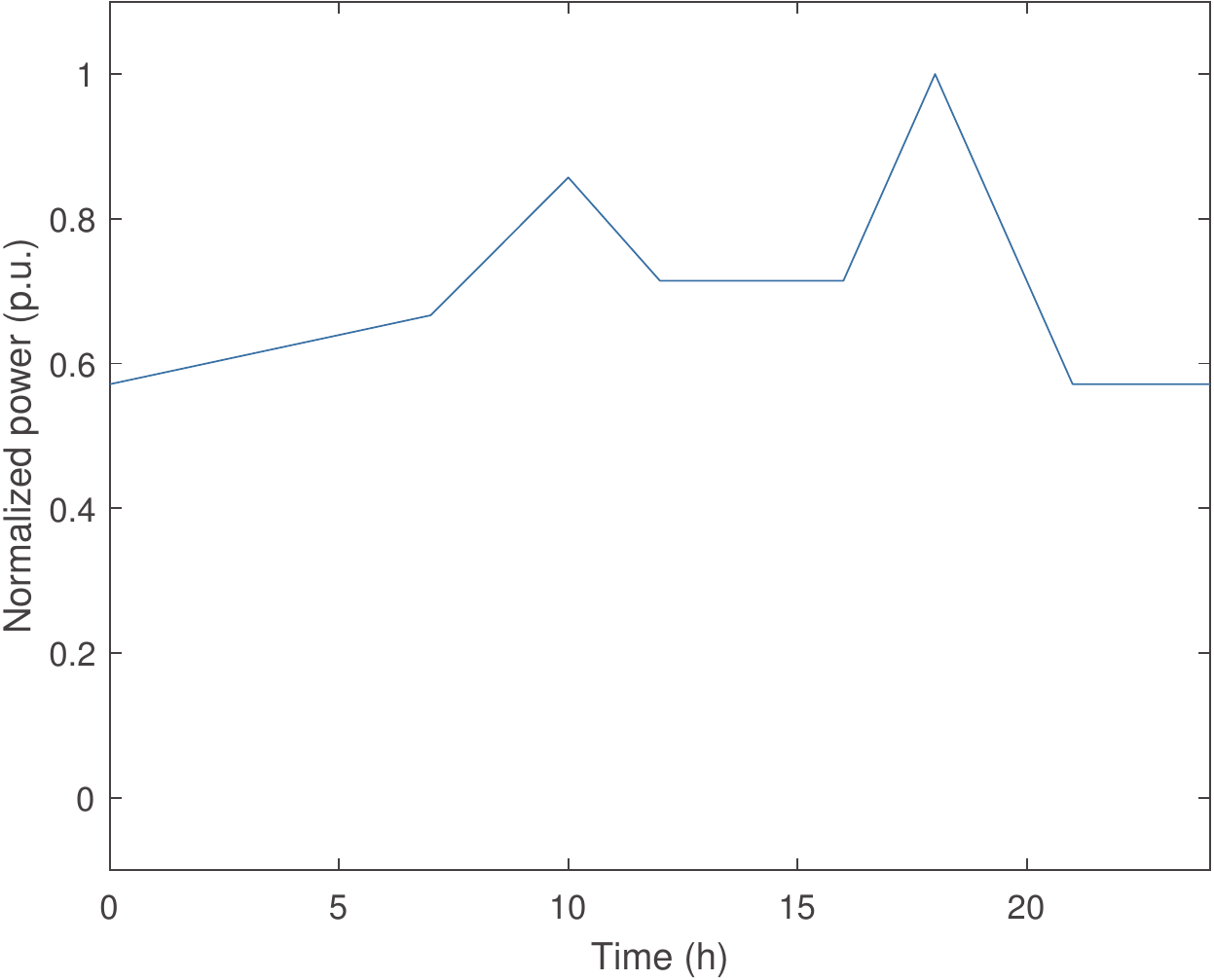}
	\caption{Normalized reference consumption profile $s^{\textrm{cons,ref}}$} \label{fig:conso}
\end{figure}

\subsection{Batteries}
{
We assume $p^{\textrm{inj}}_i$ (resp. $p^{\textrm{abs}}_i$) represents the power supplied (absorbed) by the battery at bus $i$, at time step $t \in \timeSet$, for scenario $\omega \in \Omega$.
We assume each battery can be charged/discharged in $\tau=$2 hours, i.e., $\overline{p}^{\textrm{inj}}_i = \overline{p}^{\textrm{abs}}_i = \frac{\overline{X}_i}{\tau}$ where $\overline{X}_i$ denotes the installed storage capacity at node $i$. We assume a charging efficiency $\rho^{\textrm{abs}}_i = 0.95$ while the discharging efficiency is $\rho^{\textrm{inj}}_i = 1 / \rho^{\textrm{abs}}_i$. They are the same for all buses. Additionally, we impose the periodicity constraints:
\begin{align}\label{eq:constr:cyclic:storage}
\forall i \in \busSet, \omega \in \Omega, \quad X_{i,T+1,\omega} = X_{i,1,\omega},
\end{align}
which ensure that the batteries will have the same states of charge on consecutive days. This allows us to take into account the daily repetition of the problem considered. In the light of Remark \ref{remark:generalization:storage:constraints}, this additional constraint does not jeopardize the earlier results of this paper. 
Besides, we assume the batteries cannot provide nor absorb reactive power.}


\subsection{Photovoltaic panels}
{The reactive power supplied by the solar panels $q^{\textrm{sol}}_{i,t,\omega}$ is a decision variable and we have the bound constraints:}
\begin{align*}
\forall i \in \busSet, t \in \timeSet, \omega \in \Omega, \quad - 0.3 \overline{p}^{\textrm{sol}}_i \leq q^{\textrm{sol}}_{i,t,\omega} \leq 0.
\end{align*}

\subsection{First use case: diffuse energy storage systems and solar production}\label{subsec:case:study:diffuse}

{We assume the total installed capacity of batteries is $\overline{X}^{\textrm{tot}} = 1$ MWh. 
Each bus $i \in \busSet$ is equipped with a battery with maximal energy capacity proportional to the peak load $S_i$ of the bus $\overline{X}_i = \frac{S_i}{\sum_{i \in \busSet} S_i} \overline{X}^{\textrm{tot}}$. The total installed solar capacity is denoted $\overline{p}^{\textrm{sol,tot}}$ and its value will be set to different levels later on. Each bus $i$ is equipped with photovoltaic panels with maximal capacity proportional to the size $S_i$ of the bus, i.e., $\overline{p}^{\textrm{sol}}_i = \frac{S_i}{\sum_{i \in \busSet} S_i} \overline{p}^{\textrm{sol,tot}}$. Let us show that for $\overline{p}^{\textrm{sol,tot}}$ inferior to some threshold value, \eqref{eq:bound:vlin:slin:system}-\eqref{eq:suff:cond:equal:p:restricted} hold. To this end, we consider the upper bound $\bar{s}_{i} = \frac{S_i}{\sum_{i \in \busSet} S_i} \left(\overline{p}^{\textrm{sol,tot}} + \frac{\overline{X}^{\textrm{tot}}}{\tau}\right) - 0.55 \frac{1 + \textbf{i} 0.2}{\sqrt{1+ (0.2)^2}} S_i$ on the power injections at bus $i$, valid for any time and any scenario tree, where we used the fact that $0.55 \leq \min_{t \in \timeSet} s^{\textrm{cons,ref}}_{t}$. Let us consider the linear program $(LP)$:

\begin{align*}
	\max_{\overline{p}^{\textrm{sol,tot}}, \Slinbar, \vlinbar}& \overline{p}^{\textrm{sol,tot}}\\
	s.t.\qquad & \Slinbar_{\overrightarrow{(i,j)}} = \sum_{\overrightarrow{(k,i)} \in \mathcal{E}} \Slinbar_{\overrightarrow{(k,i)}}  - 0.55 \frac{1 + \textbf{i} 0.2}{\sqrt{1+ (0.2)^2}} S_i\\
	& \qquad \qquad + \frac{S_i}{\sum_{i \in \busSet} S_i} \left(\overline{p}^{\textrm{sol,tot}} + \frac{\overline{X}^{\textrm{tot}}}{\tau}\right),  &\overrightarrow{(i,j)} \in \mathcal{E}, \\ 
	&\vlinbar_{0} = 1, \\
	&\vlinbar_{i} - \vlinbar_{j} = 2 \Real(z_{i,j}^* \Slinbar_{\overrightarrow{(i,j)}}),   &\overrightarrow{(i,j)} \in \mathcal{E}.\\
	&\vlinbar_{i} \leq \overline{v}_i,   &i \in \busSet,\\
	&\Real(z_{k,l}^* \Slinbar_{\overrightarrow{(i,j)}}) \leq 0 ,  &\overrightarrow{(i,j)} \in \mathcal{E},  \overrightarrow{(k,l)} \in \mathcal{E}_i.
\end{align*}
{The projection of the feasibility set of this optimization problem along the first component $\overline{p}^{\textrm{sol,tot}}$ is either empty, $\mathbb{R}$ if $val(LP) = + \infty$ or an interval of the form $(-\infty, val(LP)]$. By solving it numerically, we find $val(LP) = 1.7023$, which shows that for $\overline{p}^{\textrm{sol,tot}} \leq 1.7023$ MW, \eqref{eq:bound:vlin:slin:system}-\eqref{eq:suff:cond:equal:p:restricted} are satisfied and therefore, by Theorem \ref{theorem:opf:zero:duality:gap}, $(P)$ has no relaxation gap, provided that the assumptions of Theorem \ref{theorem:opf:zero:duality:gap:restricted} hold. This shows that the absence of a relaxation gap can be proved under more realistic assumptions than over-generation \cite{lava:12} or load over-satisfaction \cite{fari:13}, and in a general multi-stage stochastic setting, unlike \cite{huan:16} which considers a deterministic model. In particular, this condition does not depend on the time grid $(\tau_t)_t$ or on the scenario tree.
}
\subsection{Second use case: concentrated energy storage systems and solar production}

{In this second example, we assume that there is no battery and that only buses $7$ and $20$ are equipped with PV panels. The installed capacities are respectively denoted by $\overline{p}^{\textrm{sol}}_7$ and $\overline{p}^{\textrm{sol}}_{20}$.
We consider the upper bound $\bar{s}_{i} = \overline{p}^{\textrm{sol}}_i  - 0.55 \frac{1 + \textbf{i} 0.2}{\sqrt{1+ (0.2)^2}} S_i$ on the power injections at bus $i$, valid for any time and any scenario tree, where we used the fact that $0.55 \leq \min_{t \in \timeSet} s^{\textrm{cons,ref}}_{t}$. We want to maximize the installed solar capacity at both buses while guaranteeing an a priori vanishing relaxation gap by enforcing \eqref{eq:bound:vlin:slin:system}-\eqref{eq:suff:cond:equal:p:restricted} to hold. This yields the linear program $(LP')$:
}
\begin{align*}
\max_{\overline{p}^{\textrm{sol}}_7 ,\overline{p}^{\textrm{sol}}_{20}, \Slinbar, \vlinbar}& \overline{p}^{\textrm{sol}}_7 + \overline{p}^{\textrm{sol}}_{20},\\
s.t.\qquad & \Slinbar_{\overrightarrow{(i,j)}} = \sum_{\overrightarrow{(k,i)} \in \mathcal{E}} \Slinbar_{\overrightarrow{(k,i)}}  + \overline{p}^{\textrm{sol}}_i  - 0.55 \frac{1 + \textbf{i} 0.2}{\sqrt{1+ (0.2)^2}} S_i,  \qquad\overrightarrow{(i,j)} \in \mathcal{E}, \\ 
&\vlinbar_{0} = 1, \\
&\vlinbar_{i} - \vlinbar_{j} = 2 \Real(z_{i,j}^* \Slinbar_{\overrightarrow{(i,j)}}),   \qquad \overrightarrow{(i,j)} \in \mathcal{E}.\\
&\vlinbar_{i} \leq \overline{v}_i,   \qquad i \in \busSet,\\
&\Real(z_{k,l}^* \Slinbar_{\overrightarrow{(i,j)}}) \leq 0 ,  \qquad \overrightarrow{(i,j)} \in \mathcal{E},  \overrightarrow{(k,l)} \in \mathcal{E}_i,\\
&\overline{p}^{\textrm{sol}}_7 \geq 0,\\
&\overline{p}^{\textrm{sol}}_{20} \geq 0.
\end{align*}
{Then the maximal installed solar capacity is given by $\overline{p}^{\textrm{sol,tot}} = \overline{p}^{\textrm{sol}}_7 + \overline{p}^{\textrm{sol}}_{20} = 2.0851$ MW with $\overline{p}^{\textrm{sol}}_7 = 0.4399$ MW and $\overline{p}^{\textrm{sol}}_{20} = 1.6452$ MW. Therefore, for any values of $\overline{p}^{\textrm{sol}}_7 \leq 0.4399$ MW and $\overline{p}^{\textrm{sol}}_{20} \leq 1.6452$ MW, \eqref{eq:bound:vlin:slin:system}-\eqref{eq:suff:cond:equal:p:restricted} hold and therefore, by Theorem \ref{theorem:opf:zero:duality:gap}, $(P)$ has no relaxation gap, provided that the assumptions of Theorem \ref{theorem:opf:zero:duality:gap:restricted} hold. This is true for any choice of time grid and scenario tree, provided $s^{\textrm{cons,ref}}_{t} \geq 0.55$ for any $t \in \timeSet$.
}

\subsection{Numerical study of the upper bound on the relaxation gap}
{From now on, we numerically investigate the upper bound on the relaxation gap derived in Theorem \ref{theorem:a:posteriori:bound:relaxation:gap} on the first use case with diffuse energy storage systems and solar production, see \ref{subsec:case:study:diffuse}. We consider higher levels of installed solar capacity so that assumptions of Theorem \ref{theorem:opf:zero:duality:gap} do not hold anymore. We need to specify an optimization window and a scenario tree.}

\subsubsection{Time grid considered}

We consider an optimization window of $31$ hours, divided into $T+1 = 9$ sub-intervals. Each time step $t$ corresponds to a time interval in the model $[\tau_t, \tau_{t+1}]$. The correspondence between $t$ and $\tau_t$ is given in Table \ref{tab:time}. {One could consider instead a finer time grid, with step lengths of 15 minutes or 1 hour, which is most common in practice. However, as discussed below, the number of scenarios is generally exponential in the number of ``branching points'' of the scenario tree, hence, to avoid a blow up of the size of the optimization problem, one needs to use scenario trees branching at time steps from a coarser time grid. However, this coarse grid need not be uniformly distributed. Here, since the scenario tree provides a quantization of the randomness of solar production, we use 2-hours steps from 10 am to 6 pm, 3-hours steps from 7 to 10 am and from 9 pm to midnight the next day, and a single step of 7 hours for all the night from midnight to 7 am. We also emphasize that our theoretical results regarding the relaxation gap do not depend directly on the choice of time discretization parameters. }

\begin{table}[hbtp!]
	\centering
	\begin{tabular}{c||c|c|c|c|c|c|c|c|c|c}
		$t$ & $0$ & $1$ & $2$ & $3$ & $4$ & $5$ & $6$ & $7$ & $8$ & $9$ \\
		\hline
		Real time $\tau_t$ (h) & $0$ & $7$ & $10$ & $12$ & $14$ & $16$ & $18$ & $21$ & $24$ & $31$
	\end{tabular}
	\caption{Time steps and their corresponding time window}
	\label{tab:time}
\end{table}

\subsubsection{Generation of a scenario tree with i.i.d. scenarios for solar power production}\label{subsec:scenar:sol}

Recall that $\tau_t$ is the time associated with time step $t \in \timeSet$. 
We assume the solar power is given by:
\begin{align}\label{eq:solar:pow:decomp}
p^{\textrm{sol}}_{i,t,\omega} =\overline{p}^{\textrm{sol}}_i I^{\textrm{sol}}_{\tau_t,\omega} \overline{x}^{^{\textrm{sol,norm}}}_{\tau_t},
\end{align}
where the clear-sky index $I^{\textrm{sol}}$ is a random process taking values between $0$ and $1$, which models the clearness of the sky. Its value is $0$ when the sky is completely cloudy (i.e., even at day, there would be no light) and $1$ for a completely clear sky.
The deterministic time-dependent envelop $\overline{x}^{^{\textrm{sol,norm}}}$ models the time evolution of the normalized solar power we would observe if the sky were clear.
We suppose it is given by $\overline{x}^{^{\textrm{sol,norm}}}_{\tau} = 0$ for $\tau < T_{\textrm{day}} = 7$ and $\tau > T_{\textrm{night}}=21$ and by $$\overline{x}^{^{\textrm{sol,norm}}}_{\tau} = 0.5 - 0.5 \cos\left(\frac{2 \pi(\tau-T_{\textrm{night}})}{T_{\textrm{night}} - T_{\textrm{day}}}\right)$$
for $T_{\textrm{day}}\leq \tau \leq T_{\textrm{night}}$.


To build a scenario tree, we use the stochastic model in \cite{bado:gobe:gran:kim:17} for the clear-sky index $I^{\textrm{sol}}$, based on a Fisher-Wright-type Stochastic Differential Equation (SDE), given by:
\begin{align}\label{eq:sde:irradiance}
I^{\textrm{sol}}_{\tau} =I^{\textrm{sol}}_{0} -\int_0^{\tau} a (I^{\textrm{sol}}_s - I^{\textrm{ref}}) \dd s + \int_0^{\tau} \sigma (I^{\textrm{sol}}_s)^{\alpha}(1-I^{\textrm{sol}}_s)^{\beta} \dd B_s,
\end{align}
with $B$ a Brownian motion. The parameter $a \geq 0$ is a mean-reversion speed parameter, and $I^{\textrm{ref}} \in [0,1]$ is a reference value for the clear-sky index. Under the assumption $a \geq 0$, $\alpha,\beta \geq 0.5$, this SDE has a unique strong solution with values in $[0,1]$ almost surely, see \cite{bado:gobe:gran:kim:17}.

The branching structure of the scenario tree is characterized by pre-specified vector $(C_t)_{t \in \timeSet}$ where $C_t$ corresponds to the number of children nodes of a node at stage $t$. In particular, the total number of scenarios is given by $N = \prod_{t=1}^{T} C_t$. 
Given a structure of a scenario tree, we use the quantile-based Algorithm \ref{algo:scenario:tree} with parameters given in Table \ref{tab:parameters:sde:algo} to instantiate the values of solar irradiance at the nodes of the scenario tree. 
Each scenario is assigned probability $1/N$, which is consistent with the fact that we consider evenly spaced quantiles for the values of successors of each node.

\begin{table}[h!]
	\centering
	\begin{tabular}{c||c|c|c|c|c|c|c|c}
			Parameter &$I^{\textrm{ref}}$ & $a$ & $\sigma$ & $\alpha$ & $\beta$ & $I^{\textrm{sol}}_{0}$ & $M$ & $\tau^{\textrm{Euler}}$\\
			\hline
			Value & $0.75$ & $0.75 \ h^{-1}$ & $0.7$ &  $0.8$ & $0.7$ & $0.5$ & $10000$ & $ 0.1 h$
	\end{tabular}
\caption{Parameters of SDE of solar irradiance $I^{\textrm{sol}}$ and of Algorithm \ref{algo:scenario:tree}\label{tab:parameters:sde:algo}}
\end{table}

\begin{algorithm}[h!]
	\caption{Generation of scenario tree of solar irradiance $I^{\textrm{sol}}$}
	\label{algo:scenario:tree}
	\begin{algorithmic}
		{\small
			\STATE Given: $M$, $(C_t)_{t = 1,...,T}$, $\{\tau_1,...,\tau_T\}$, $\tau^{\textrm{Euler}}$.\\
			\STATE $val$: table of values of the nodes of scenario tree
			\STATE $val[n_1] \leftarrow I^{\textrm{sol}}_{init}$ with $n_1$ node at time $t=1$
			\FOR{$t= 2,...,T$}
			\FOR{$n_{t-1}$ node at stage $t-1$}
			\STATE Simulate $M$ i.i.d. trajectories of $I^{\textrm{sol}}$ in \eqref{eq:sde:irradiance} on $[\tau_{t-1}, \tau_t]$ conditionally to $I^{\textrm{sol}}_{\tau_{t-1}} = val[n_{t-1}]$ using Euler scheme with step $\tau^{\textrm{Euler}}$.
			\FOR {$i = 1...C_{t-1}$}
			\STATE Create $i^{th}$ child of node $n_{t-1}$, denoted $n_t$.
			\STATE Set $val[n_t]$ to quantile $\frac{100(2i-1)}{2C_{t-1}} \%$ of simulated values of $I^{\textrm{sol}}_{\tau_t}$.
			\ENDFOR
			\ENDFOR
			\ENDFOR
			\STATE Return $val$.
		}
	\end{algorithmic}
\end{algorithm}


We consider various scenario trees which approximate the stochastic process $I^{\textrm{sol}}$, where branching (i.e., $C_t >1$) occurs at the time steps where $\overline{x}^{\textrm{sol, norm}}_t$ is big. This is heuristically justified by the fact that instantaneous volatility of the solar power at time $t$ is directly proportional to $\overline{x}^{^{\textrm{sol,norm}}}_{\tau_t}$. The scenarios of normalized solar power $I^{\textrm{sol}} \overline{x}^{^{\textrm{sol,norm}}}$ are represented in Figure \ref{fig:scenarios:solar} for three scenario trees with respectively 1 scenario {($C_t = 1$ for all $t \in \timeSet$)}, 8 scenarios {($C_t =2$ for $\tau_t \in \{10,12,14\}$ and $C_t = 1$ else)} and 12 scenarios {($C_t =2$ for $\tau_t \in \{10,14\}$, $C_t = 2$ for $\tau_t = 12$ and $C_t = 1$ else)}.

\begin{figure}[h!]
	\centering
	\begin{subfigure}{ 0.3 \textwidth}
		\includegraphics[width = \textwidth]{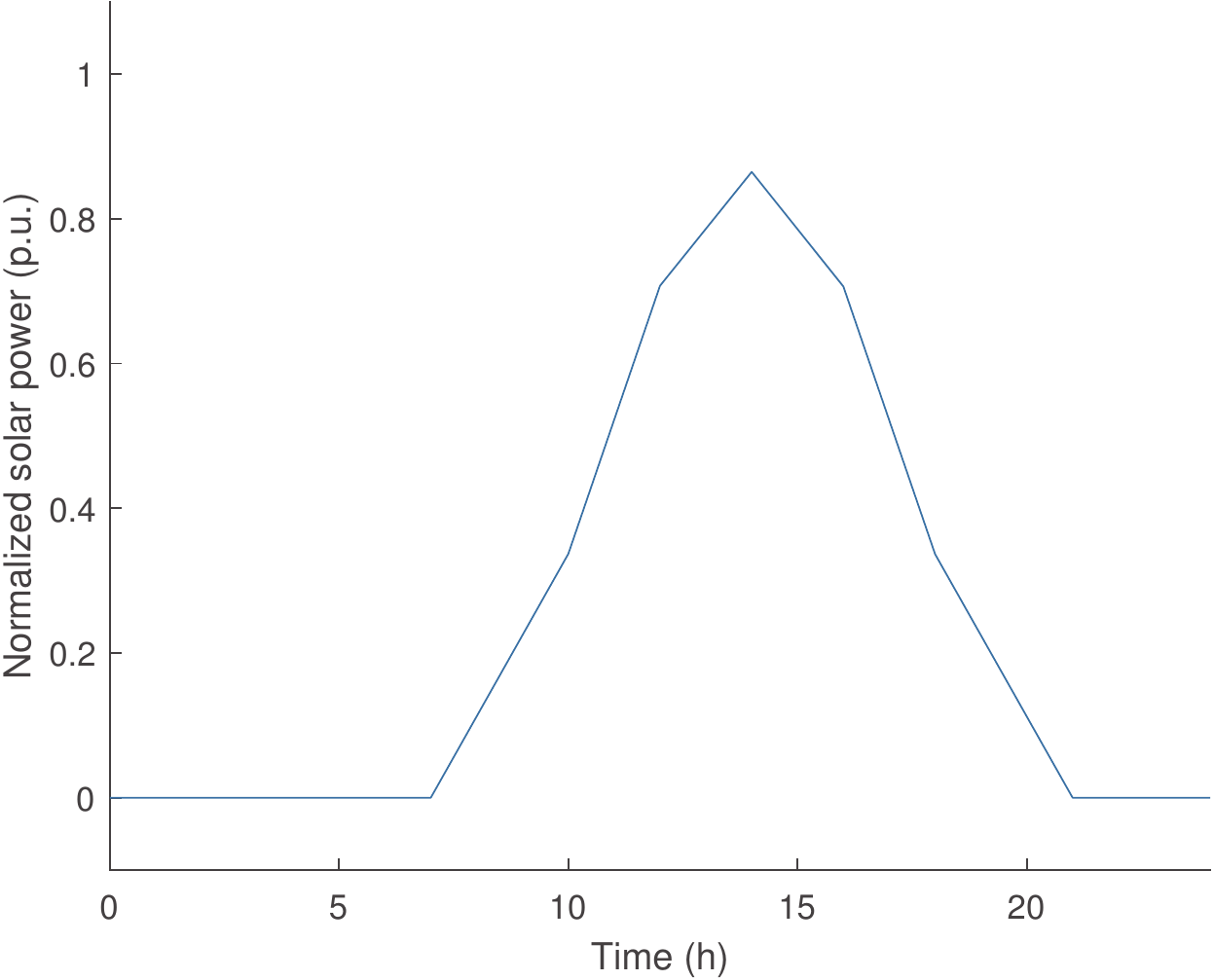}
	\end{subfigure}
	\begin{subfigure}{ 0.3 \textwidth}
		\includegraphics[width = \textwidth]{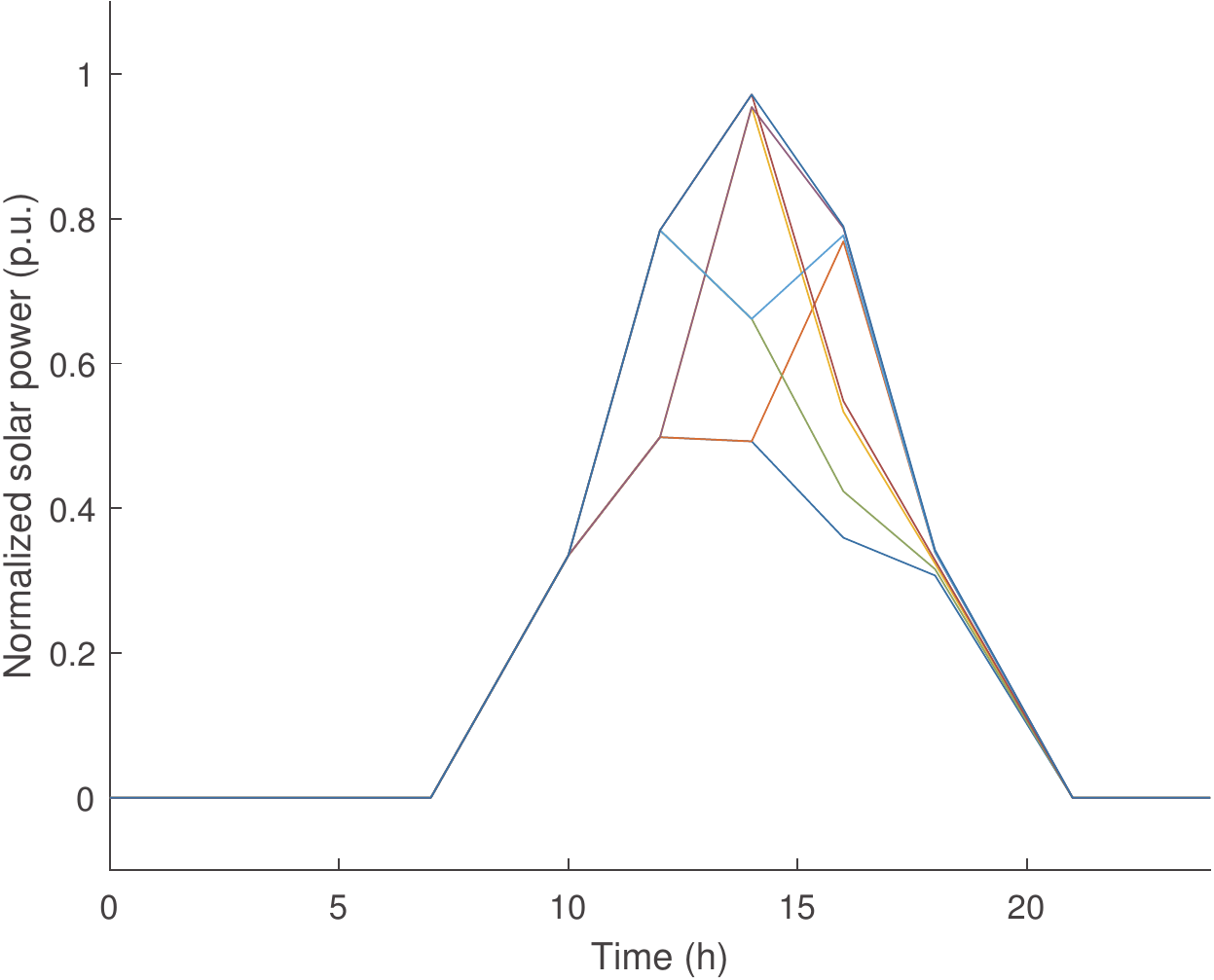}
	\end{subfigure}
	\begin{subfigure}{ 0.3 \textwidth}
		\includegraphics[width = \textwidth]{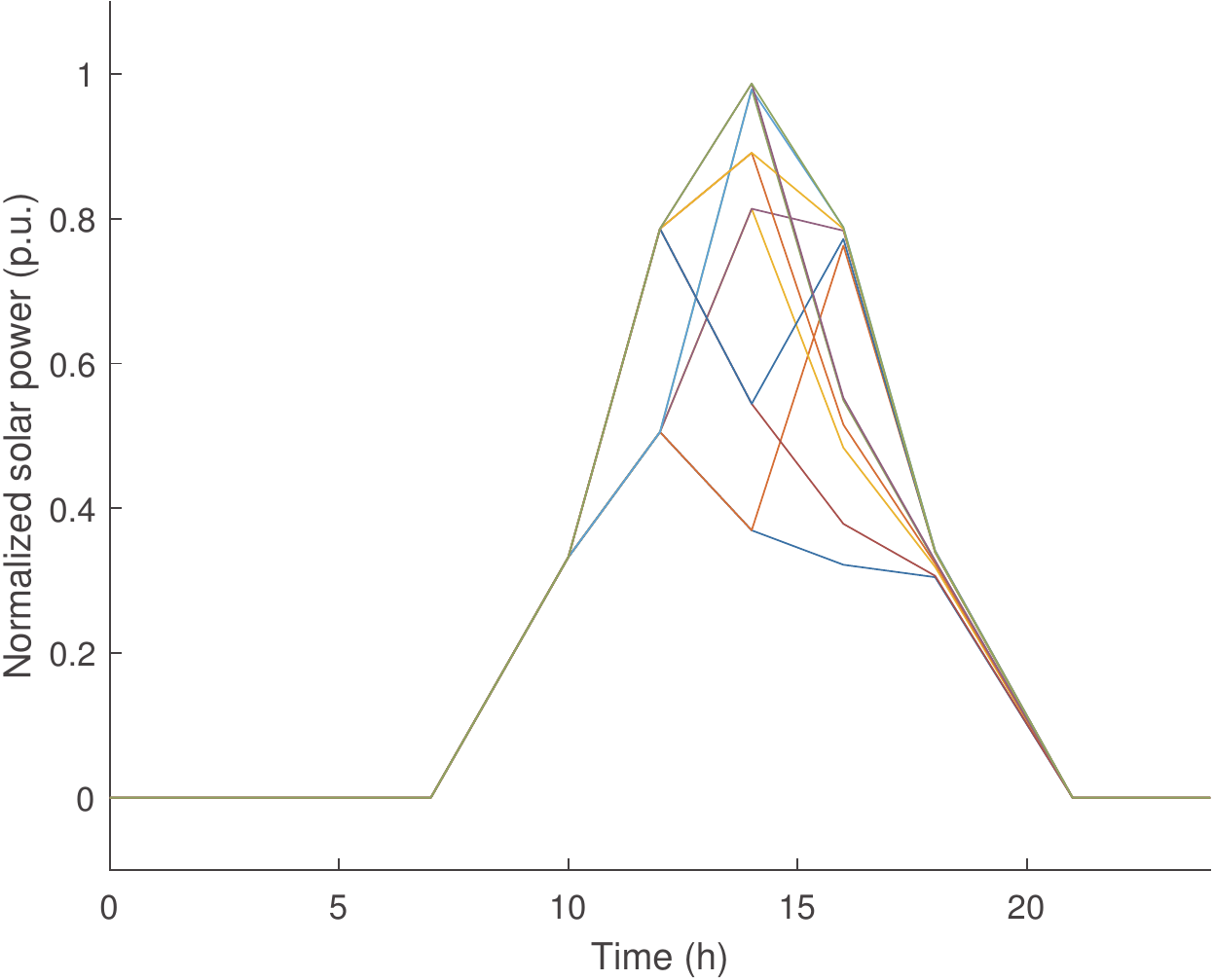}
	\end{subfigure}
	\caption{Scenario trees of daily normalized solar power production $\overline{x}^{^{\textrm{sol,norm}}} I^{\textrm{sol}}$ ($1$, $8$ and $12$ scenarios)} \label{fig:scenarios:solar}
\end{figure}

\subsubsection{Numerical results}

We assume that the cost functional is given by:
\begin{align}\label{eq:opf:cost:1}
&\frac{1}{N} \sum_{\omega =1}^N \sum_{t=0}^T \Delta_t\left(c_{0,+} (p_{0,t,\omega})_+  - c_{0,-} (p_{0,t,\omega})_-  +\sum_{\overrightarrow{(i,j)} \in \mathcal{E}}  c_{\textrm{loss}} r_{i,j} \mathcal{I}_{\overrightarrow{(i,j)},t ,\omega} \right) ,
\end{align}
{where $c_{0,+} = 1 \ \textrm{MW}^{-1}$ represents a marginal cost of importing electricity in the feeder considered from the public grid (connected to the feeder at bus $0$), $c_{0,-} = 0.5 \ \textrm{MW}^{-1}$ respectively represents a marginal gain when sending electricity back to the public grid and $c_{\textrm{loss}} = 2 \ \textrm{MW}^{-1}$ represents the marginal cost of thermal losses in the distribution network. 
}
We do not incorporate storage costs, which would require an estimation using a technical and economical analysis \cite{swam:17} or a specific mathematical model \cite{carp:chan:dela:riga:19}.


We consider the value $\overline{p}^{\textrm{sol,tot}} = 3$ MW. For this value, one cannot invoke Theorem \ref{theorem:opf:zero:duality:gap} guaranteeing a vanishing relaxation gap. Instead, we use Theorem \ref{theorem:a:posteriori:bound:relaxation:gap} to compute an a posteriori bound on the relative relaxation gap $\epsilon$ defined by:
\begin{align}\label{eq:bound:relative:relax:gap}
\epsilon = \frac{2 \left(val(P'_{\textrm{SOC}}) - val(P_{\textrm{SOC}})\right)}{\left|val(P_{\textrm{SOC}})\right| + \left|val(P'_{\textrm{SOC}})\right|}.
\end{align}

Table \ref{tab:opf:num:results:56} gives the bound on the relative relaxation gap as a function of the number of scenarios. Computation times corresponding to the optimization are also reported.

\begin{table}[h!]
	\centering
	{\begin{tabular}{c||c|c|c}
			number of scenarios N & $1$ & $8$ & $12$ \\
			\hline
			Bound on relative relaxation gap $\epsilon$ & $0$ & $4.5 \times 10^{-8}$ &  $1.3 \times 10^{-6}$ \\
			Optimization time $(P_{\textrm{SOC}})$ & $1.79$ s & $38.9$ s & $165.9$ s \\
			Optimization time $(P'_{\textrm{SOC}})$ & $1.72$ s & $50.3$ s & $153.2$ s
	\end{tabular}}
\caption{Bound on relative relaxation gap $\epsilon$ and computation time depending on the number of scenarios $N$ for the 56 buses network\label{tab:opf:num:results:56}}
\end{table}


For the three scenario trees considered, the bound on relative relaxation gap of $(P)$ is zero (up to a numerical tolerance). {The probability distributions of active power losses and power injections at bus $0$ over time of the solution of $(P_{\textrm{SOC}})$ (which is also feasible for $(P)$) are represented in the form of box-plots in Figures \ref{fig:boxplot:losses} and \ref{fig:boxplot:p0} respectively. We can notice than power losses and injections at bus $0$ decrease during the day, owing to the local solar production.}

\begin{figure}[h!]
	\centering
	\begin{subfigure}{ 0.3 \textwidth}
		\includegraphics[width = \textwidth]{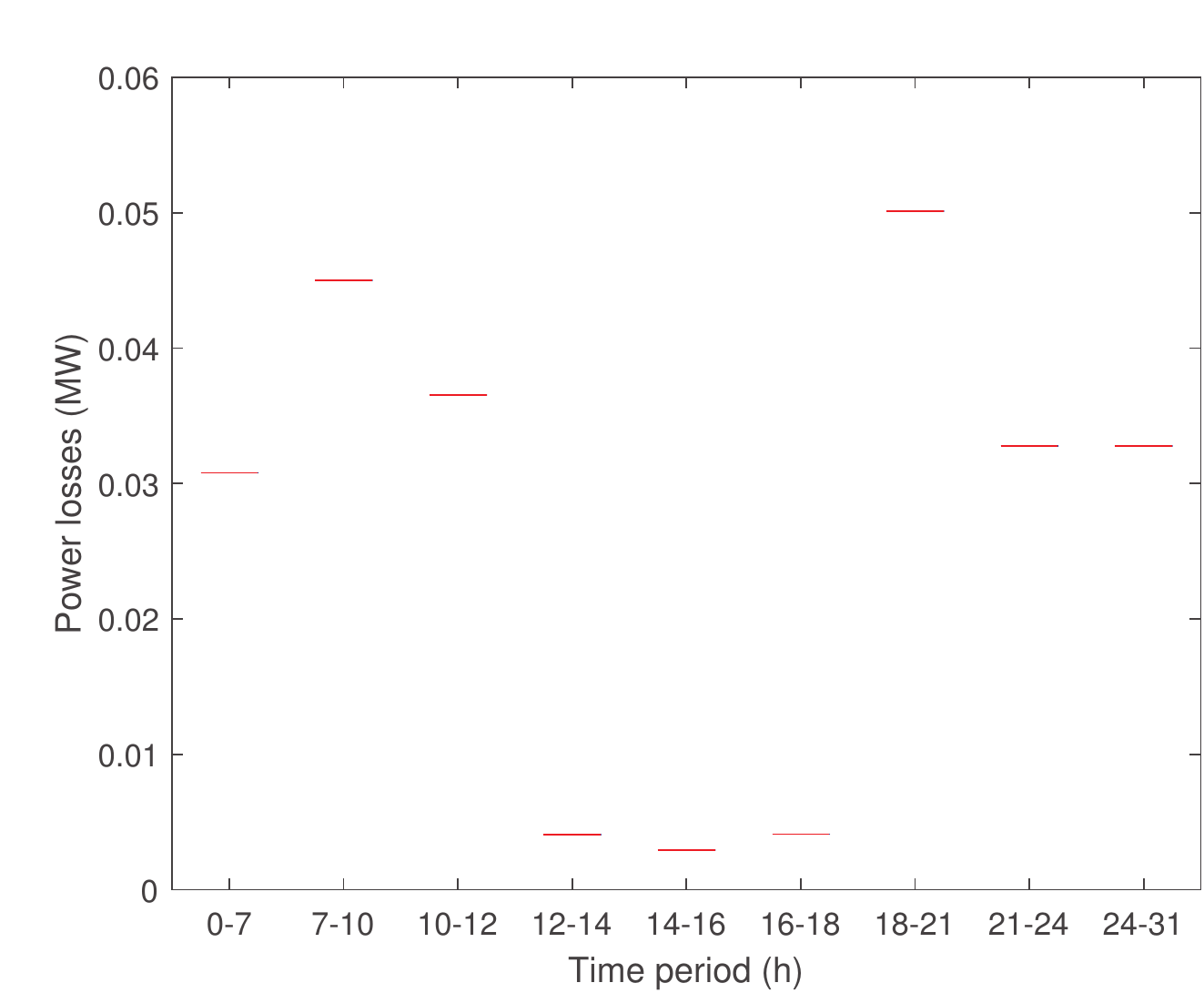}
		\subcaption{1 scenario}
	\end{subfigure}
	\begin{subfigure}{ 0.3 \textwidth}
		\includegraphics[width = \textwidth]{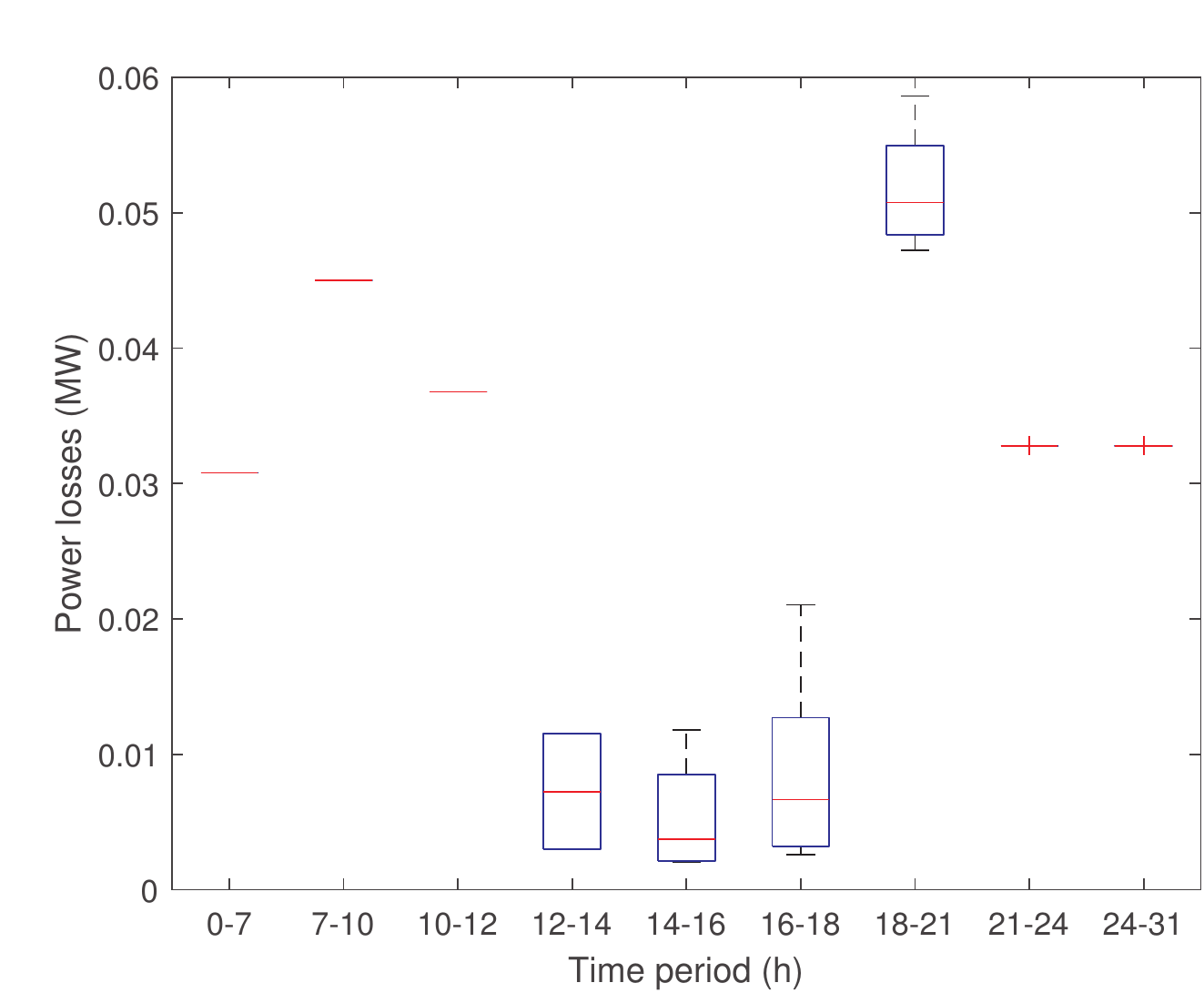}
		\subcaption{8 scenarios}
	\end{subfigure}
	\begin{subfigure}{ 0.3 \textwidth}
		\includegraphics[width = \textwidth]{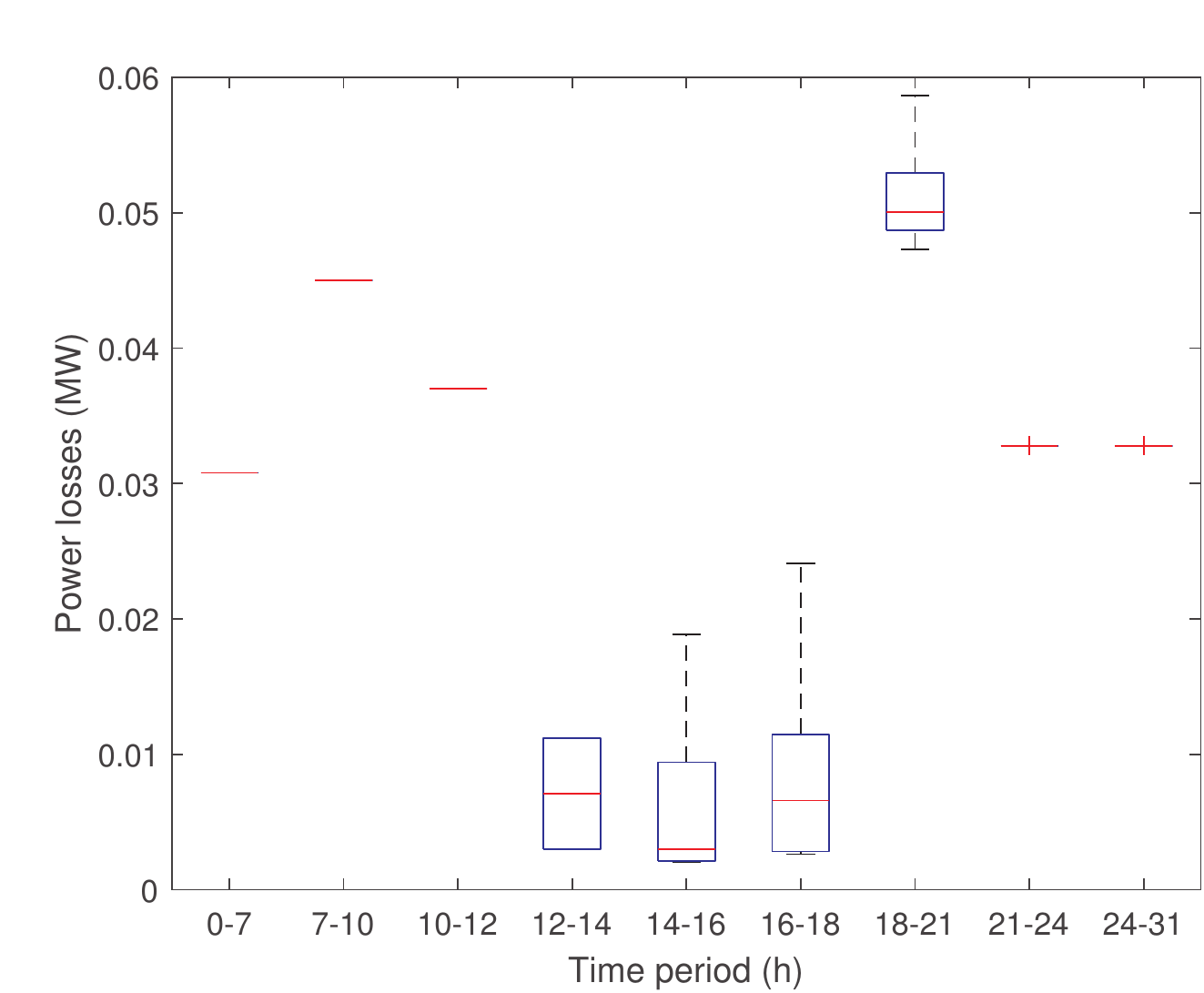}
		\subcaption{12 scenarios}
	\end{subfigure}
	\caption{Empirical distribution of total active losses in the network over time} \label{fig:boxplot:losses}
\end{figure}

\begin{figure}[h!]
	\centering
	\begin{subfigure}{ 0.3 \textwidth}
		\includegraphics[width = \textwidth]{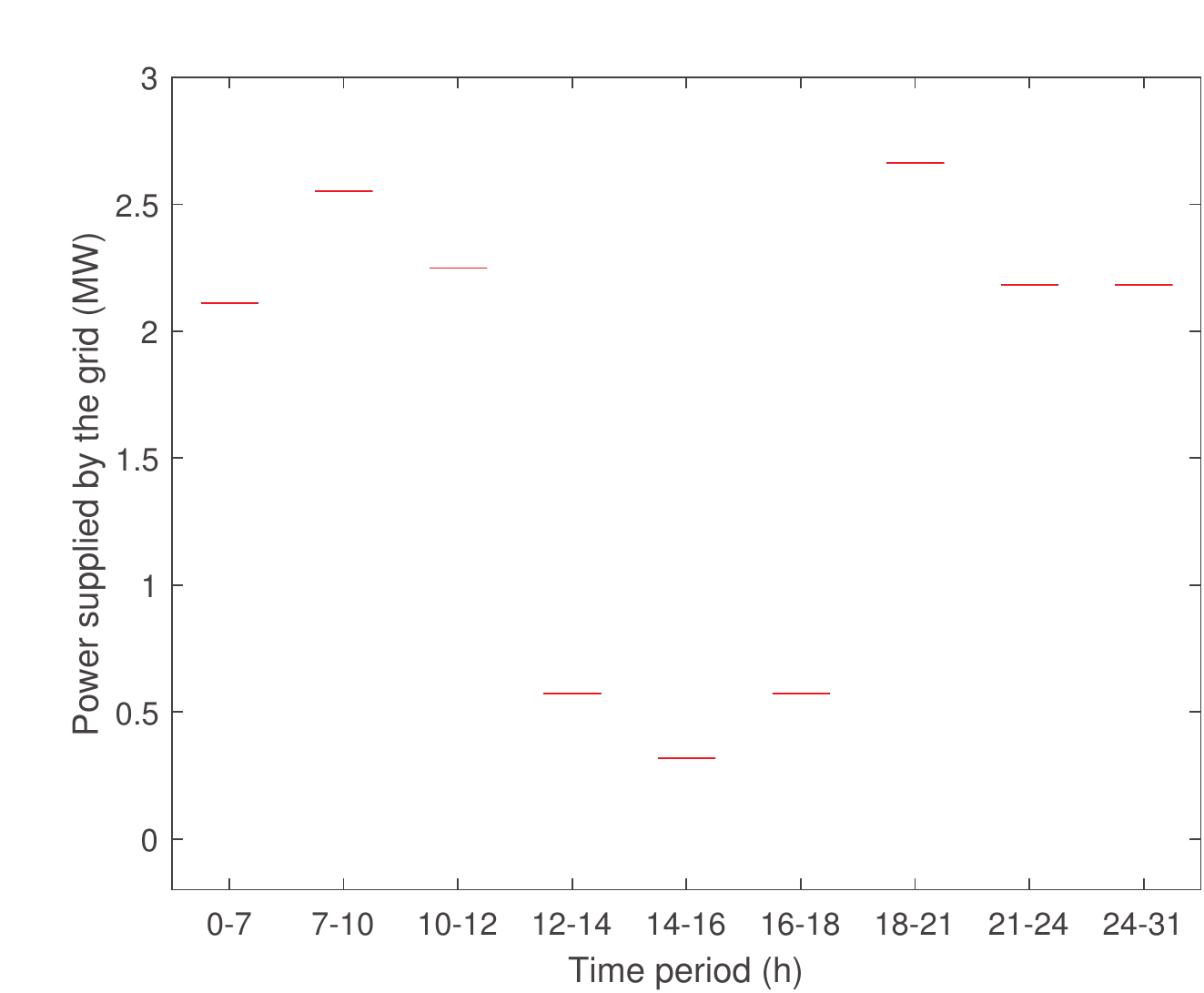}
		\subcaption{1 scenario}
	\end{subfigure}
	\begin{subfigure}{ 0.3 \textwidth}
		\includegraphics[width = \textwidth]{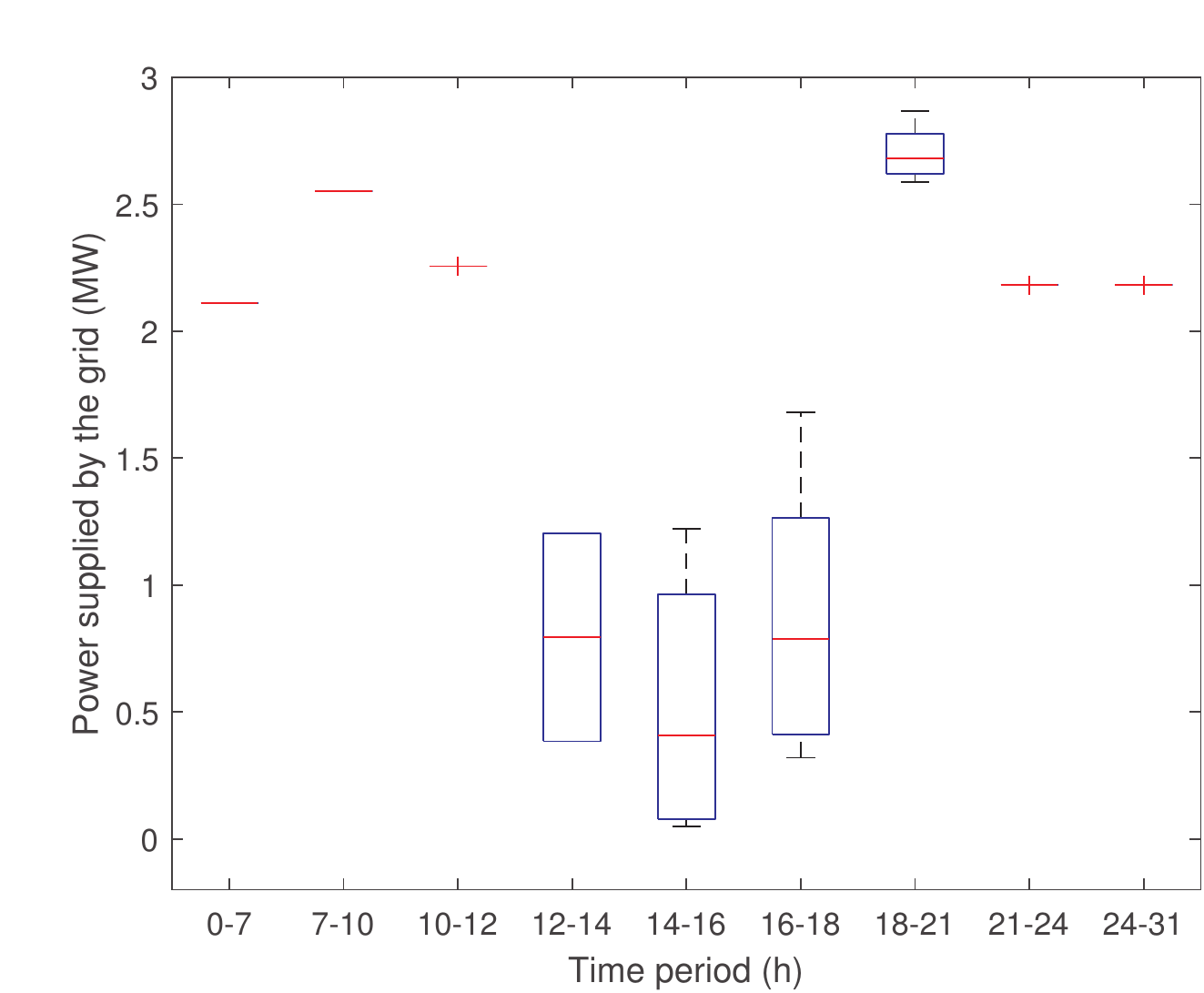}
		\subcaption{8 scenarios}
	\end{subfigure}
	\begin{subfigure}{ 0.3 \textwidth}
		\includegraphics[width = \textwidth]{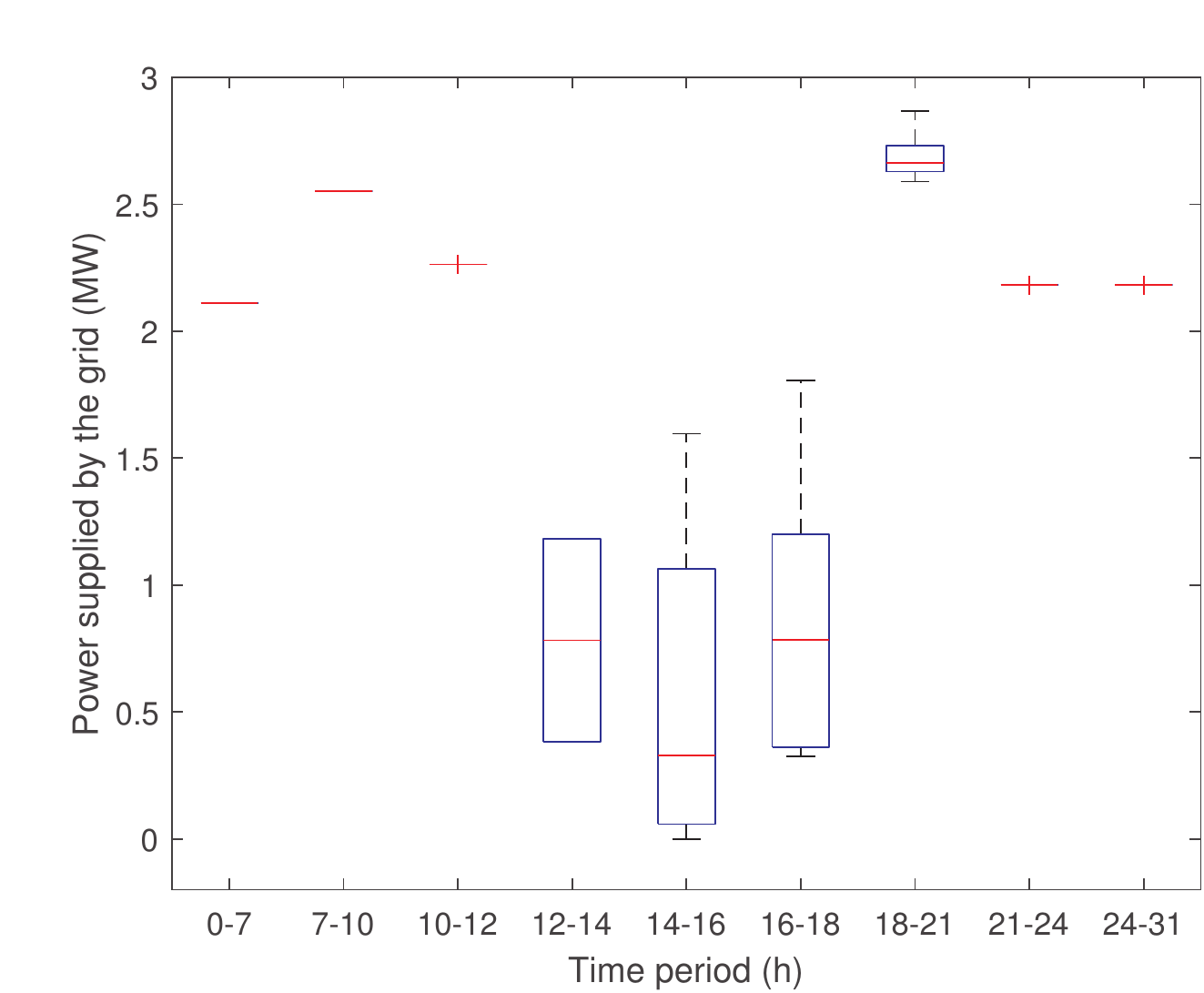}
		\subcaption{12 scenarios} 
	\end{subfigure}
	\caption{Empirical distribution of active power injections at bus $0$ over time} \label{fig:boxplot:p0}
\end{figure}

We now exhibit typical cases where the a posteriori bound on the relaxation gap can be expected to be good (i.e., close to the true value of the relaxation gap) or not. To this end, we consider a deterministic case with different values for the storage cost $c_{\textrm{bat}}$ and the total installed solar capacity $\overline{p}^{\textrm{sol,tot}}$. We consider the cost functional:
\begin{align}\label{eq:opf:cost:2}
&\sum_{t=0}^T \Delta_t\left(c_{0,+} (p_{0,t})_+  - c_{0,-} (p_{0,t})_- +  \sum_{\overrightarrow{(i,j)} \in \mathcal{E}} c_{\textrm{loss}}  r_{i,j} \mathcal{I}_{\overrightarrow{(i,j)},t}  + \sum_{i \in \busSet}c_{\textrm{bat}} (p^{\textrm{inj}}_{i,t} + p^{\textrm{abs}}_{i,t}) \right).
\end{align}

We give in Table \ref{tab:opf:num:results:2} the values of upper bound on the relative relaxation gap $\epsilon$ defined in \eqref{eq:bound:relative:relax:gap}, as a function of the installed solar capacity and storage cost.

\begin{table}[h!]
	\centering
	\begin{tabular}{c||c|c|c|c|c}
		$\overline{p}^{\textrm{sol,tot}}$ & $1.5$ MW & $3$ MW & $3.5$ MW & $4$ MW & $4.5$ MW  \\
		\hline
		$c_{\textrm{bat}} = 0 \ \textrm{MW}^{-1}$ & $0$ & $0$ & $7.7 \times 10^{-6}$ &  $4.0 \times 10^{-4}$ & $+\infty$ \\
		$c_{\textrm{bat}}=1 \ \textrm{MW}^{-1}$ & $3.7 \times 10^{-8}$ & $0$ & $7.2 \times 10^{-4}$ & $6.2 \times 10^{-3}$ & $+\infty$ \\
		$c_{\textrm{bat}}=2 \ \textrm{MW}^{-1}$ & $0$ & $3.4 \times 10^{-7}$ & 
		$7.2\times 10^{-4}$ & $3.1 \times 10^{-2}$ & $+\infty$
	\end{tabular}
	\caption{A posteriori bound $\epsilon$ depending on storage cost and installed solar capacity (56 buses network)}
	\label{tab:opf:num:results:2}
\end{table}


One can show that for all the instances considered in Table \ref{tab:opf:num:results:2}, the optimal solution of $(P_{\textrm{SOC}})$ found is feasible for $(P)$, showing that the relaxation gap is zero in all cases investigated: $val(P) = val(P_{\textrm{SOC}})$. {Therefore, considering problem $(P'_{\textrm{SOC}})$ is not necessary to estimate the relaxation gap or to prove that it is null, but it allows to assess the quality of the bound $\epsilon$ on the relative relaxation gap: we aim at finding typical situations where the bound $\epsilon$ is close to the true value of the relative relaxation gap $0$, and situations where this is not the case. The bound $\epsilon$ is close to $0$ if the} installed solar capacity is low {($1.5$ MW)} in agreement with Theorem \ref{theorem:opf:zero:duality:gap} or intermediate {($3 MW$)}. For higher values of the installed solar capacity {($4$ MW)}, the a posteriori bound is of better quality for lower storage cost. This makes sense since the batteries can be leveraged to absorb power to avoid reverse power flow at low cost (guaranteeing \ref{eq:constr:BFM:additional:constr:multistage}). {Figure \ref{fig:battery:power} confirms this observation: it represents the total active power supplied by the storage systems over time for the optimal solution of the restricted problem $(P'_{\textrm{SOC}})$ for $c_{\textrm{bat}} = 2 \ \textrm{MW}^{-1}$ (to minimize battery use). When $\overline{p}^{\textrm{sol,tot}} = 3.5$ MW, the batteries are not used at all whereas for $\overline{p}^{\textrm{sol,tot}} = 4.5$ MW, the storage systems need to be employed in order to ensure that \eqref{eq:constr:BFM:additional:constr:multistage} hold.} For higher values of the installed solar capacity, the a posteriori bound $\epsilon$ becomes even infinite, although the relaxation gap of $(P)$ is still zero. In that case, the storage capacity is not sufficient to absorb power, and hence constraint \ref{eq:constr:BFM:additional:constr:multistage} cannot be satisfied, even using batteries. The numerical results show that the a posteriori bound on the relaxation gap is good in the case of low decentralized production capacity or if active or reactive power can be absorbed locally at low cost.

\begin{figure}[h!]
	\begin{subfigure}{ 0.49 \textwidth}
		\includegraphics[width = \textwidth]{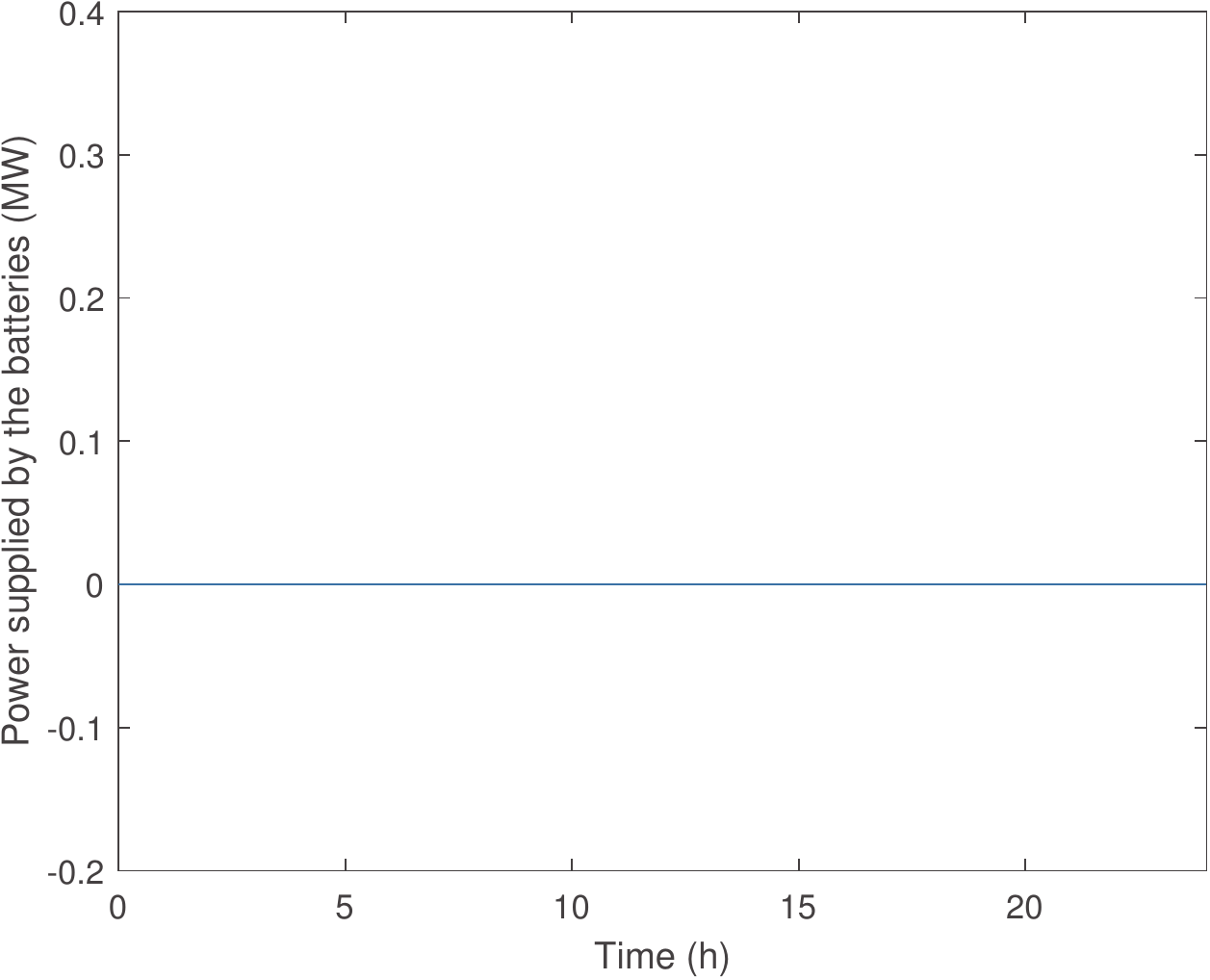}
		\subcaption{$\overline{p}^{\textrm{sol,tot}} = 3.5$ MW}
	\end{subfigure}
	\begin{subfigure}{ 0.49 \textwidth}
		\includegraphics[width = \textwidth]{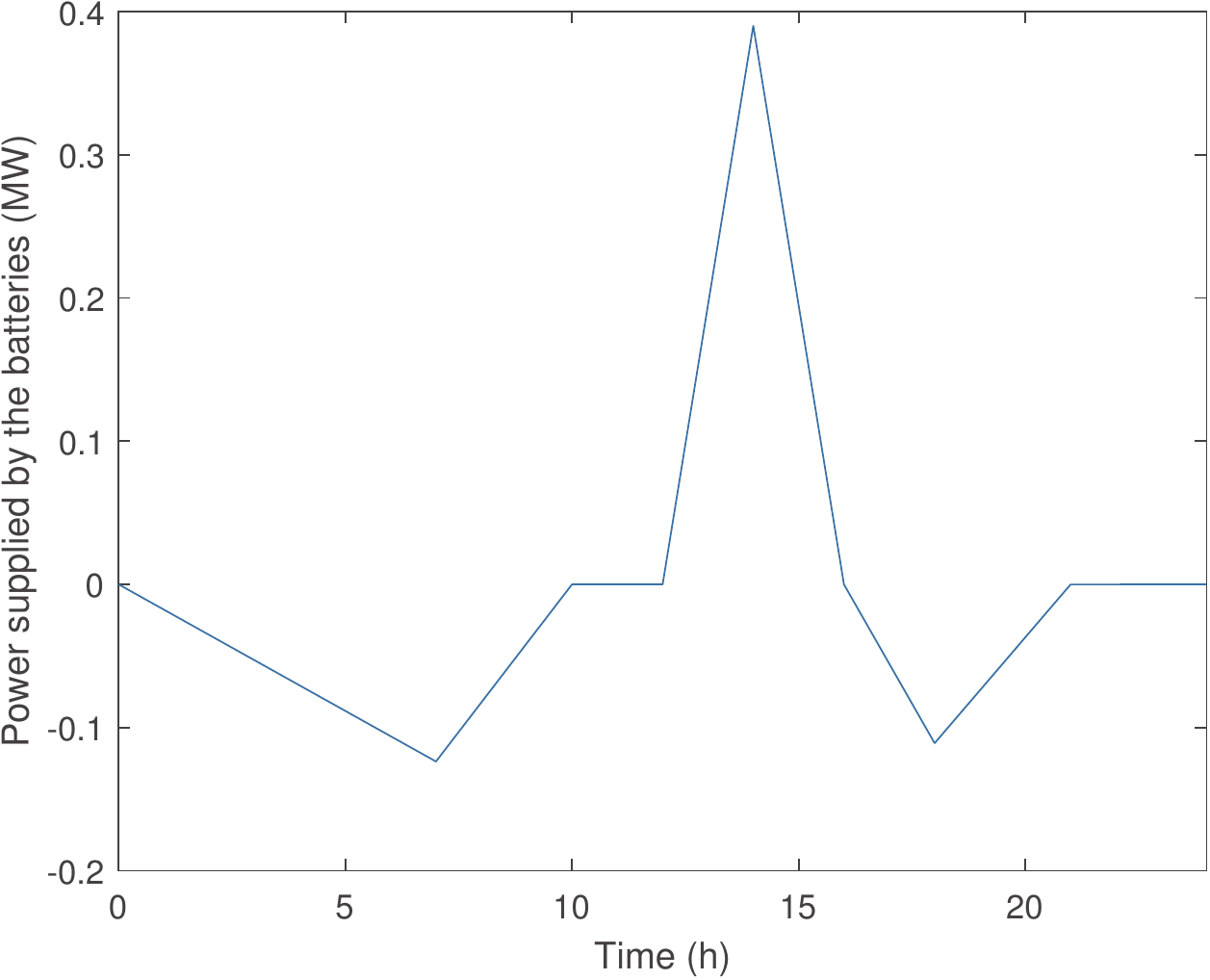}
		\subcaption{$\overline{p}^{\textrm{sol,tot}} = 4$ MW}
	\end{subfigure}
	\caption{Evolution of total active power supplied by batteries for $c_{\textrm{bat}} = 2 \ \textrm{MW}^{-1}$} \label{fig:battery:power}
\end{figure}

\section{Conclusion}
We have introduced a multi-stage stochastic AC OPF problem to account for both uncertainty of renewable production and dynamic constraints of storage systems. We have extended a result ensuring a vanishing relaxation gap under some realistic a priori conditions for the static deterministic AC OPF problem to the multi-stage stochastic setting. A similar procedure also yields an easily computable a posteriori upper bound on the relaxation gap of the problem.
These results are illustrated by numerical experiments on a realistic distribution network with local generation and storage systems. {Possible extensions of this work would be to consider three-phase unbalanced networks, controllable VRT and capacitor banks.}

\section*{Acknowledgments}
This work has benefited from several supports: Siebel Energy Institute (Calls for Proposals \#2, 2016), ANR project CAESARS (ANR-15-CE05-0024), Association Nationale de la Recherche Technique (ANRT), Electricit\'e De France (EDF), \emph{Finance for Energy Market} (FiME) Lab (Institut Europlace de Finance). The authors would like to thank Riadh Zorgati, Mathieu Caujolle and Bhargav Swaminathan for fruitful discussion. They also thank the reviewers for their detailed comments.

\bibliography{ArticleMicrogridMain}

\end{document}